\documentclass[11pt]{article} 
\usepackage[margin=1in]{geometry}

\usepackage{authblk}
\usepackage{amsthm}
\usepackage{amssymb}
\usepackage{amsmath}
\usepackage{float}
\usepackage{tikz}
\usetikzlibrary{arrows,calc,shapes,decorations.pathreplacing}
\usepackage{caption}
\usepackage{pgfplots}
\usepackage{subcaption}
\usepackage{enumitem}
\usepackage[toc,page]{appendix}

\usepackage{algorithm}
\usepackage{algorithmicx}
\usepackage{algpseudocode}

\usepackage{bm} 
\usepackage{color} 

\def\addlegendimage{\csname pgfplots@addlegendimage\endcsname}

\newtheorem{definition}{Definition}
\newtheorem{theorem}{Theorem}
\newtheorem{lemma}{Lemma}
\newtheorem{proposition}{Proposition}

\newtheorem{conjecture}{Conjecture}
\newtheorem{assumption}{Assumption}

\theoremstyle{remark}

\newtheorem{remark}{Remark}

\newcommand{\vc}{\bm}
\newcommand{\calC}{\mathcal{C}}
\newcommand{\calT}{\mathcal{T}}

\newcommand{\calN}{\mathcal{N}}
\newcommand{\1}{\mbox{1}\hspace{-0.25em}\mbox{l}}
\newcommand{\ol}{\overline}
\newcommand{\bbR}{\mathbb{R}}

\def\P{\mathrm{P}}
\def\E{\mathrm{E}}

\DeclareMathOperator*{\argmin}{arg\,min}

\newcommand*\diff{\mathop{}\!\mathrm{d}}
\newcommand{\parrow}{\:\to_{\mathbb P}\:}

\begin{document}

\title{Strategic arrivals to a queue with service rate uncertainty\footnote{To appear in Queueing systems: Theory and Applications.}}

\author[1,2]{Liron Ravner}
\author[3]{Yutaka Sakuma}
\affil[1]{\small{University of Amsterdam}}
\affil[2]{\small{Eindhoven University of Technology}}
\affil[3]{\small{National Defense Academy of Japan}}

\date{\today}
\maketitle

\begin{abstract}
We study the problem of strategic choice of arrival time to a single-server queue with opening and closing times when there is uncertainty regarding service speed. A Poisson population of customers choose their arrival time with the goal of minimizing their expected waiting times and are served on a first-come first-served basis. There are two types of customers that differ in their beliefs regarding the service time distribution. The inconsistent beliefs may arise from randomness in the server state along with noisy signals that customers observe. Customers are aware of the two types of populations with differing beliefs. We characterize the Nash equilibrium dynamics for exponentially distributed service times and show how they substantially differ from the model with homogeneous customers. We further provide an explicit solution for a fluid approximation of the game. For general service time distributions we provide an algorithm for computing the equilibrium in a discrete time setting. We find that in equilibrium customers with different beliefs arrive during different (and often disjoint) time intervals. Numerical analysis further shows that the mean waiting time increases with the coefficient of variation of the service time. Furthermore, we present a learning agent-based model (ABM) in which customers make joining decisions based solely on their signals and past experience. We numerically compare the long-term average outcome of the ABM with that of the equilibrium and find that the arrival distributions are quite close if we assume (for the equilibrium solution) that customers are fully rational and have knowledge of the system parameters, while they may greatly differ if customers have limited information or computing abilities.
\end{abstract}

\section{Introduction}\label{sec:Intro}

A bottleneck arrival game is typically modeled as a queue that operates during a specified period of time and has to serve a finite population of incoming customers who can choose their arrival time. In the queueing literature this is known as the ?/M/1 model introduced by Glazer and Hassin in \cite{GH1983}. The characteristics of a congested bottleneck queue are often not the same on different days and customers arriving at the bottleneck may not be aware of the current state on a given day. For example, the duration of an airport security check can differ from day to day and location to location depending on many variables such as staffing issues and specific security alerts. Moreover, customers may have different beliefs regarding the speed of service due to past experience and private information obtained from external sources such as the media or friends. Customers can choose when to arrive to the queue with the goal of minimizing their expected waiting times, but the uncertainty regarding the service times implies that customers are heterogeneous when evaluating expected waiting times. We study the simplest model for this situation by assuming that there are two possible service states, ``fast" and ``slow", and two distinct populations of customers each believing that the state is one of the possibilities. We assume that the total number of customers arriving at any given day is a Poisson random variable and analyze both continuous time and discrete times cases.

A Nash equilibrium is given by a set of, possibly mixed, arrival strategies for every type of customer belief. For the case of exponential service times we characterize the Nash equilibrium arrival dynamics of both customer types and show that solving them is a much more challenging task than in the standard model with homogeneous customer beliefs (e.g., \cite{HK2011}). Partial results for the equilibrium arrival strategies are derived along with a conjecture on a general solution. The intractability of the continuous time Markovian model leads us to consider a deterministic fluid approximation of a large scale system (many customers with very small jobs) which can be solved explicitly. We further consider a discrete time system with general service times where customers can only be admitted into the queue at specified time slots. If the time slots are very close then this can approximate the continuous time setting, but the model also captures the dynamics of a system with limited admission times (e.g., a day with three admission periods: morning, afternoon and evening). Furthermore, moving away from exponential service times enables the examination of the effect of the variance of the service time distribution. This framework is useful because it is very general, but explicit analysis seems intractable and so we present an algorithmic procedure for the computation of the Nash equilibrium arrival distributions. For the continuous and discrete time settings we show that there are multiple types of equilibrium arrival distributions, including everyone arriving at the opening, one class of customers at the beginning and the other according to a mixed distribution during the working hours, or both arriving according to a mixed distribution.

We further present a random service model with noisy signals that may result in inconsistent customer beliefs due to information asymmetry. Suppose that the server state is assigned randomly every day according to a fixed probability $p\in[0,1]$ (for slow service) and a signaling mechanism reveals the true state with a high probability $q>\frac{1}{2}$. This model is similar to the setting of the join or balk game studied in \cite{DV2014}. This results in seemingly inconsistent customer beliefs regarding the service time distribution that correspond to the posterior distributions given the signal observed. Therefore, through this service uncertainty model and signal mechanism we have a game that is in fact consistent with an assumption of homogeneous customers from a Bayesian point of view. It is important to point out that after receiving signal customers make a posterior update of the service time distribution, and therefore even if the original distribution was exponential the posterior distribution is no longer exponential. Therefore, the numerical procedure for computing the equilibrium for a general service time distribution in the discrete time setting is very useful.

In order to compute the optimal arrival strategies the customers must be aware of all of the system parameters, including the characteristics of customers with different beliefs and the quality of the signal they receive. Furthermore, the customers must be able to use the available information to make posterior updates regarding the sizes of both population types, knowing that some received a wrong signal, and their respective service time distributions. However, in some settings customers may not be aware of all of the relevant system parameters and of the quality of the signal. This leads us to explore several assumptions on the information available to customers and their level of rationality. In particular we consider both the equilibrium in a fully rational setting and also the case where customers do not make posterior update, be it due to lack of information regarding the parameters or due to limited computing ability. We refer to the latter as a bounded rationality assumption. Moreover, we introduce a dynamic learning agent-based model (ABM) where customers update their estimators for the expecting waiting time given their signal and actual delay in the system on previous days with the same signal. The learning procedure is a mixture between randomly testing new time slots (exploration) and choosing a time slot that was optimal in the past on average (exploitation) such that the probability of exploration decreases to zero with time.

The equilibrium in the discrete setting is numerically compared to the dynamic learning model. In equilibrium, for both the fully rational and bounded rationality cases, the customers with different types of beliefs arrive on mostly disjoint intervals. This observation is also backed up by the explicit solution of the fluid model. As the ABM does not assume customer knowledge of the system parameters, i.e., the arrival rates, service time distributions, true state probability $p$ and signal quality $q$, there is an inherent estimation bias of the expected waiting time. This is because on a proportion of $1-q$ of the days the signal is wrong and the delay observations are misclassified. However, the long-term outcome of the ABM appears to be much closer to the full information equilibrium solution than to the bounded rationality equilibrium solution. We further show that this bias generally makes the arrival distributions more spread out, which increases the average waiting times compared to the full information equilibrium, but to a lesser extent than the increased waiting time arising from the equilibrium solution that ignores the signal quality. In fact, in some cases social welfare is higher for the completely ignorant customers in the ABM than in the equilibrium solution when customers make the wrong assumptions (in the bounded rationality case).

\vspace{2mm}
\noindent {\bf Main contributions.} 
We now summarize the main contributions of our work.
\begin{itemize}
\item \textit{Heterogeneous customers:} Thus far, the literature on bottleneck queue arrival games has mostly assumed homogeneous customers. This paper is the first to analyze a game with a discrete (non-fluid) population and customers with heterogeneous beliefs regarding the service time distribution.
\item \textit{Asymmetric information:} We present a simple model for a bottleneck queue arrival game with service quality uncertainty and provide equilibrium analysis for both the Markovian and non-Markovian settings. Our framework further enables studying different levels of available information and customer rationality.
\item \textit{Discrete game with large time slots:} For the discrete game we provide a framework to analyze the equilibrium when the acceptance period is divided into big time slots such that multiple jobs can be served within a time slot. This is more general than discrete systems with small time slots that aim to approximate continuous time, and the resulting equilibrium distributions can be quite different. 
\item \textit{Bounded rationality and learning:} Our framework enables the comparison of equilibrium outcomes with various assumptions regarding the information availability and the level of rationality of the customers. The equilibrium arrival distributions are further compared with the outcome of a learning process in which customers update their decisions based on past experiences by estimating the expected waiting time corresponding to every pair of time choice and signal. This allows us to examine the effect of the biased estimation in the learning process as opposed to the limited and full information equilibrium solutions.
\end{itemize}

\vspace{2mm}
\noindent {\bf Background and related literature.}
The research on strategic queueing deals with many aspects of decision making in queueing systems. An introduction and overview of this research can be found in \cite{book_HH2003} and \cite{book_H2016}. The ?/M/1 model of Glazer and Hassin \cite{GH1983} was the first to consider an endogenous arrival process to a queue that is determined by strategic considerations of the arriving customers. The server was assumed to start working in time zero and customers choose their arrival time with the goal of minimizing expected waiting costs. It was assumed that customers can queue before the server commences operations, an assumption which was later removed in \cite{HK2011} that considered a finite acceptance period. The latter established that the equilibrium arrival distribution has an atom at zero, an interval with no arrivals and a continuous distribution on a continuous interval. The socially optimal arrival distribution was numerically analyzed and it was shown that it is optimal to allow arrivals on a discrete set of times. The model was later extended to include tardiness penalties in \cite{JS2013} and \cite{H2013}, order penalties in \cite{R2014}, a loss system in \cite{HR2015}, and a network of queues in \cite{HJ2015}. The issues of existence and uniqueness of equilibrium were addressed in detail in \cite{JS2013}. In \cite{B2017} the model was generalized to general service times, i.e., ?/G/1, and it was shown that the equilibrium has the same form and uniqueness properties, but is much harder to compute the arrival distribution. A fluid approximation for a game with heterogeneous customers that differ in their cost functions was presented in \cite{JS2018}. The discrete-time version of the ?/G/1 game was analyzed in \cite{SMF2019}, which also presented an algorithm to compute the equilibrium. This discrete-time setting can be used as an approximation for the continuous time model when taking a large number of small time slots. Moreover, \cite{SMF2019} presented the dynamic learning model which is extended here to allow multiple customer types. A more detailed review of the literature on arrival time games can be found in Chapter 4.1 of \cite{book_H2016}. 

Observable queues with customers deciding to join or balk based on noisy signals regarding the service speed and value were studied in \cite{VD2009} and \cite{DV2014}. A queueing game with inconsistent customer beliefs regarding the service rate of an M/M/1 system was studied in \cite{CV2016}. The authors considered the problem of joining or balking an observable queue when the service rate beliefs of customers are generated by a common continuous distribution; the optimal strategy is derived and shown to have a threshold form, where the threshold depends on the individual belief of the customer. In \cite{HC2015} a model with anecdotal reasoning of customers was introduced: before deciding whether to join an unobservable queue customers obtain a noisy signal in the form of a (single) past experience of an acquaintance who joined the queue. The equilibrium outcome was compared with a fully rational benchmark. The impact of uncertainty regarding server capacity on staffing and control of multiserver queues was investigated in \cite{I2018}. 

\vspace{2mm}
\noindent{\bf Paper organization.}
The remainder of the paper is organized as follows. Section \ref{sec:model} introduces the queueing model and game-theoretic framework. analyzes the continuous time game with exponential service times. In particular, it highlights the difficulty of solving the equilibrium dynamics even in the simplest Markovian setting which can be solved by standard methods in the case of homogeneous beliefs. In Section \ref{sec:fluid} we present the explicit equilibrium solution for a fluid model that can be viewed as an approximation of a large scale system with many small customers that have negligible service times. Section \ref{sec:discrete} defines the stochastic discrete-time game and provides a numerical method to compute the equilibrium distributions, with numerical analysis of the method given in \ref{sec:numerical}. Section \ref{sec:learn} presents the agent-based model and discusses the estimation bias associated with it. Furthermore, the latter also provides numerical analysis and comparison of the discrete game and agent-based model. Section \ref{sec:conclusion} concludes the paper with a discussion of the results and potential future research.

\section{Model and Preliminaries}\label{sec:model}

A population of $N\sim$Poisson$(\lambda)$ customers need to arrive at a queue during an acceptance period $\mathcal{T}\subseteq[0,T]$ such that $0\in\mathcal{T}$. Time $t=0$ is referred to as the {\em opening time}. The acceptance period may be a continuous interval $\mathcal{T}=[0,T]$ or a discrete collection of times $\mathcal{T}=\{0,t_1,\ldots,t_n\}$. All customers arriving during the acceptance period are served, even if service is completed after the acceptance period. The jobs in the queue are processed FCFS and if two, or more, jobs arrive at the same time then they are uniformly ordered at random. 

There are two types of customers $\mathcal{C}:=\{a,b\}$ that differ in their belief regarding the service time distribution $X$. Type $a$ customers are pessimistic and believe the service rate is ``slow" and type $b$ customers are optimistic and believe the service rate is ``fast". Further denote $\lambda_i:=\lambda \alpha_i$, where $\alpha_i\in(0,1)$ is the proportion of customers with belief $i\in\mathcal{C}$. We next provide a formal definition for a customer belief and the assumption that distinguishes the two types of service.   
\begin{definition}\label{def:belief}
A customer holds belief $i\in\mathcal{C}$ if his expected waiting time for any chosen arrival strategy is given by that of a queueing process with iid services times distributed as the random variable $X_i$ with cumulative distribution function (cdf, for short) $G_i$, together with the arrival process defined by the strategies of all customers.
\end{definition}
\begin{assumption}\label{assump:slow}
The expected service time corresponding to belief $a$ is higher than that corresponding to belief $b$; $\chi_a:=\E X_a>\E X_b=:\chi_b$.
\end{assumption}
\noindent{\em Examples:}
\begin{enumerate}
\item Exponential service times with a slower service rate for type $a$, $\mu_a< \mu_b$.
\item Deterministic service times: $\chi_a > \chi_b$.
\item Discrete-time memoryless service: geometric service times with success probabilities $\frac{1}{\chi_a}<\frac{1}{\chi_b}$.
\end{enumerate}

The reasoning for inconsistent beliefs is information asymmetry regarding the service capacity of the system. This may be due to different past experience. For example, some customers may be repeat customers that have an estimate of the capacity based on their experience while others may be joining for the first time and only rely on some common knowledge regarding the system. Another option is a system with random capacity where some customers acquire a noisy signal regarding the true system state. In Section \ref{sec:signals} a random environment and signal mechanism are introduced such that the system state, slow or fast, is random and the customers receive noisy signals regarding the true state. Throughout the paper we use the terms ``customer with belief $i$'' and ``type $i$ customer'' synonymously.

Customers wish to avoid waiting and choose their arrival time to the queue with the goal of minimizing expected waiting time. An arrival strategy for a customer is given by a probability distribution on $\mathcal{T}$ and can be represented by a cdf $F(t)$. We denote the support of an arrival strategy $F$ by $\sigma(F) := \inf\{ \mathcal{S} \subseteq \mathcal{T} s.t. \int_{s \in \mathcal{S}} F(ds) = 1 \}$. We focus our analysis on the symmetric (within types) case where all customers of type $i\in\mathcal{C}$ use the same strategy $F_i$. 

Let $A(t)$ denote the cumulative arrival process of customers to the system. Given an arrival profile $(F_a,F_b)$, the amount of customers arriving during an interval $(s,t]$ is a Poisson random variable,
\[
A(t)-A(s)\sim{\rm Poisson}(\lambda_a (F_a(t)-F_a(s))+\lambda_b (F_b(t)-F_b(s)))\ .
\]
This is verified by applying the Poisson splitting property for any interval $[s,t)$, which further implies that the increments on disjoint intervals are independent. Hence, if $F_a,F_b$ are continuous distributions then the arrival process is a nonhomogeneous Poisson process with cumulative rate $\lambda_a F_a(t)+\lambda_b F_b(t)$ (see, e.g., Chapter 2 of \cite{book_DVJ2003}). Observe that all customers assume the same arrival process, regardless of their individual belief, because it does not depend on the service times.

The assumption of a Poisson number of customers is convenient for several reasons. First of all, as explained above, it implies a Markovian (nonhomogeneous in time) arrival process. This makes the analysis of the queueing bottleneck game easier than the seemingly simple case of a deterministic number of arrivals, which is in fact as difficult to analyze as the case of a general random variable (see \cite{JS2013,HR2015}). Furthermore, this assumption is appropriate in a setting where each customer in a big population arrives with a small probability, due to the Poisson approximation of a Binomial random variable. Another attractive property of the Poisson number of players is that under standard axiomatic assumptions, for any tagged customer the number of other customers joining, conditional on the event that the tagged customer joins, is also Poisson random variable with the same mean (e.g., \cite{HM2012}).

Let $W_i(t)$ denote the waiting time for a type $i\in\mathcal{C}$ customer arriving at $t$. The distribution of the waiting time depends on the specific assumptions made on the service distribution, and we will consider several cases in this paper. Customers wish to arrive at a time $t$ that minimizes their expected waiting time
\begin{align*}
\label{eqn:}
 w_i(t;F_a,F_b):=\E[W_i(t) | \text{ Types $a$ and $b$ choose $F_a$ and $F_b$, respectively}]
\end{align*}
We are interested in stable arrival profiles $(F_a,F_b)$ such that no customer can decrease her/his expected waiting time by acting unilaterally.

\begin{definition}
\label{def:NE}
A symmetric (within types) Nash equilibrium is given by a pair of distributions $(F_a,F_b)$ such that for $i\in\mathcal{C}$,
\[
w_i(t;F_a,F_b)=w_i(s;F_a,F_b), \quad \forall s,t\in\sigma(F_i)\ , 
\]
and
\[
w_i(t;F_a,F_b)\leq w_i(s;F_a,F_b), \quad \forall t \in\sigma(F_i), s\in\mathcal{T} \ .
\]
\end{definition}

It is important to note that a customer with belief $i$ assumes that the service time distribution of \textbf{all} customers is $X_i$, regardless of their types (or beliefs). This is in sharp contrast to a system with two types of customers with different service distributions. Therefore, the expected waiting time $w_i(t;F_a,F_b)$ corresponds to a different queueing process for the two types of customers $\{a,b\}$. For example, in the continuous time setting, type $a$ customers believe the queue is an $M_t/G_a/1$ system while type $b$ customer believe it is an $M_t/G_b/1$ system. Furthermore, the ``true'' distribution of the service times (and queueing process) may be one of the two beliefs or neither, and it is not relevant for the equilibrium calculations as customers make their decisions according to their designated belief. It is further assumed that customers are aware of the information asymmetry and thus of the presence of customers with a belief that is inconsistent with their own.

\subsection{Signal mechanism}\label{sec:signals}

This section illustrates how a system with random environment and noisy signals can lead to inconsistent beliefs regarding the service speed, even if the customers are otherwise homogeneous. On a given day, the server operates in mode $a$ with probability $p$, otherwise, it operates in mode $b$ with probability $1-p$. Let $M$ denote the service mode, i.e., its probability distribution is given by
\begin{align*} 
\P(M=a) = p, \quad \P(M=b) = 1-p.
\end{align*}
The service time in mode $i$ ($i \in \calC = \{a,b\}$) is denoted by a random variable $X_{i}$ with mean $\chi_{i}$. As before, we assume that $\chi_{a} > \chi_{b}$, i.e., the service time in mode $a$ is slower than one in mode $b$. For every customer $k$ let $S^{(k)} \in \{0,1\}$ be a Bernoulli random variable denoting whether the true mode is observed or not ($1$ indicating the true mode). We assume that $S^{(k)}$'s are iid with $\P(S^{(k)}=1) = q$, where $q\in(\frac{1}{2},1]$, and that they are also independent of $M$. Let $Y^{(k)}$ denote the signal received by customer $k$, then
\begin{align*} 
Y^{(k)}=M S^{(k)}+(a+b-M)(1-S^{(k)})\ .
\end{align*}
The marginal probability for customer $k$ to receive a type $i\in\mathcal{C}$ signal is
\begin{equation*}\label{eq:prob_type}
\P(Y^{(k)}=i)=\left\lbrace \begin{array}{cc}
pq+(1-p)(1-q),& i =a\ , \\
p(1-q)+(1-p)q ,& i =b\ .
\end{array}\right.
\end{equation*}
Let $\alpha_{ji}:= \P(Y^{(h)}=j|Y^{(k)}=i)$ denote the conditional probability that any other customer $h$ ($h \not= k$) has a signal $Y^{(h)}=j$ given that customer $k$ received signal $i$ for $(i,j)\in\mathcal{C}^2$. Then we have
\[
\alpha_{ji} = 
\frac{\P(\left\lbrace M S^{(h)}+(a+b-M)(1-S^{(h)})=j\right\rbrace\cap\left\lbrace \ M S^{(k)}+(a+b-M)(1-S^{(k)})=i\right\rbrace )}{\P( M S^{(k)}+(a+b-M)(1-S^{(k)})=i)}\ ,
\]
and therefore,
\begin{equation}
\label{eqn:Bias_tagged_class_a}
\alpha_{ja}=\left\lbrace \begin{array}{cc}
\frac{pq^2+(1-p)(1-q)^2}{pq+(1-p)(1-q)},& j =a \ , \\
\frac{pq(1-q)+(1-p)(1-q)q}{pq+(1-p)(1-q)} ,& j =b \ ,
\end{array}\right.
\end{equation}
and
\begin{equation}
\label{eqn:Bias_tagged_class_b}
\alpha_{jb}=\left\lbrace \begin{array}{cc}
\frac{p(1-q)q+(1-p)q(1-q)}{p(1-q)+(1-p)q},& j =a \ , \\
\frac{p(1-q)^2+(1-p)q^2}{p(1-q)+(1-p)q} ,& j =b \ .
\end{array}\right.
\end{equation}
The total number of customers is a Poisson random variable with mean $\lambda$. Therefore, for customer $k$ with signal $Y^{(k)}=i\in\mathcal{C}$ the number of customers with the same signal is Poisson with mean $\lambda\alpha_{ii}$ and the number of customers with the opposite signal $j\neq i$, is $\lambda\alpha_{ji}$. 
The mean population size seen by customer $k$ with signal $i \in \calC$ is given by $\vc{\nu}_{i} :=  (\lambda\alpha_{ai}, \lambda\alpha_{bi})$.

The posterior service time distribution of customer $k$ with signal $Y^{(k)}=i \in \calC$ is given by the mixture $\tilde{X}_i^{(k)} = X_j$ with probability $\eta_{ji} := \P(M=j | Y^{(k)} = i)$ for $j \in \calC$, where $\eta_{ji}$ is given by
\[
(\eta_{aa}, \eta_{ba}) := \frac{(pq, (1-p)(1-q))}{pq + (1-p)(1-q)}, \quad 
(\eta_{ab}, \eta_{bb}) := \frac{(p(1-q), (1-p)q)}{p(1-q) + (1-p)q}\ .
\]
Moreover, the posterior expected service time of customer $k$ is
\begin{equation}
\label{eqn:Biased_service_time_class_i}
\zeta_i=\E[\tilde{X}_i^{(k)}]=\E[X_{M}|Y^{(k)}=i]=\left\lbrace \begin{array}{cc}
\frac{\chi_a pq+\chi_b (1-p)(1-q)}{ pq+(1-p)(1-q)},& i =a \ , \\
\frac{\chi_a p(1-q)+\chi_b (1-p)q}{ p(1-q)+ (1-p)q} ,& i =b \ .
\end{array}\right.
\end{equation}
Let $\tilde{X}_a$ and $\tilde{X}_b$ denote generic random variables for $\tilde{X}_a^{(k)}$ and $\tilde{X}_b^{(k)}$, respectively. The signal mechanism results in two populations of customers with distinct beliefs on both service time and mean population sizes given by $(\tilde{X}_a, \vc{\nu}_{a})$ and $(\tilde{X}_b, \vc{\nu}_{b})$.

\subsection{Beliefs, information and bounded rationality}\label{sec:belief}

In the prequel it was assumed that customers can fully compute all of the above posterior distributions, and implicitly that they know $p$ and $q$. If the only information available to them is that $q>\frac{1}{2}$ then $\zeta_i=\chi_i$ is a reasonable assumption. However, there needs to be additional information or some prior belief in order to determine $\lambda_a$ and $\lambda_b$. In other cases the customers may have no knowledge about the uncertainty mechanism and simply have prior beliefs according to their past experience or the anecdotal experiences of other customers.

Throughout the paper we consider several different assumptions of customer rationality and knowledge.
\begin{enumerate}
\item[(FR)] \textit{Fully informed rational customers:} Customers know all of the system parameters and can compute the respective posterior distributions $(\tilde{X}_a,\tilde{X}_b)$. For a customer with signal $a$ (resp. $b$), the posterior service-time distribution is given by $\tilde{X}_a$ (resp. $\tilde{X}_b$), and the mean size of populations is given by $\vc{\nu}_{a}$ (resp. $\vc{\nu}_{b}$).
\item[(BR)] \textit{Customers with bounded rationality or limited information:} Customers either do not know $p$ or $q$, or are unable to make the posterior distribution computations. Customers trust their signal (assuming $q>\frac{1}{2}$). The service times of the beliefs are given by $(X_a,X_b)$ with means $(\chi_a,\chi_b)$. Customers further have prior beliefs regarding the population sizes $\lambda_a$ and $\lambda_b$ which may, or may not, correspond the posterior service rates computed in the prequel.
\item[(AM)] \textit{Agent-based model}: Customers are not maximizing expected utility but rather optimize their behavior by learning the system behavior from repeated experiences. Every day they receive a signal regarding the server state and make an arrival time choice. After every day customers update their estimated delay for every time period given the signal obtained.
\end{enumerate}

Note that (FR) and (BR) both yield a Nash equilibrium, even though the customers may not be fully rational. Assumption (AM), however, does not assume any kind of equilibrium outcome. Section \ref{sec:exp} considers a fully Markovian model which corresponds to (BR): every type of customer has a different belief regarding the rate of the exponential service times. A deterministic fluid approximation of the system is presented in Section \ref{sec:fluid}. The fluid model can be used to approximate either (FR) or (BR) and yields an explicit solution for any given parameters. The discrete-time model of Section \ref{sec:discrete} enables numerical analysis of both (FR) and (BR) in a stochastic setting. In our numerical analysis of the discrete time model we consider several examples with service times that follow deterministic, geometric, and a mixture of geometric distributions; where the first two cases correspond to (FR) and the latter to (BR) (for a Markovian system with discrete geometric service times). A model building on Assumption (AM) is introduced in Section \ref{sec:learn}. The outcomes of the different models are compared numerically in Sections \ref{sec:numerical}, \ref{sec:NE_with_bias_vs_ABM} and \ref{sec:NE_no_bias_vs_ABM}.

\section{Exponential service times}\label{sec:exp}

This section discusses the properties and dynamics of the symmetric (within type) Nash equilibrium arrival profiles for the case of exponential service times and a continuous acceptance period $[0,T]$. Before proceeding to the characterization of the equilibrium arrival profile we first review the properties of the homogeneous beliefs model. 

\subsection*{Homogeneous beliefs}

Suppose for now that all customers share the common belief that the exponential parameter of the service time distribution is $\mu$. The number of customers is a Poisson random variable with mean $\lambda$. In this case the game reduces to the ?/M/1 queue of \cite{GH1983,HK2011}. Let $F$ denote the symmetric arrival cdf and consider a tagged customer that wants to choose an arrival time that minimizes their expected delay. Denote the total number of arrivals until time $t$ by $A_F(t)$. As explained in Section \ref{sec:model}, due to the Poisson splitting property, we have that $A_F(t)\sim$Poisson$(\lambda F(t))$. Further denote by $Q_F(t)$ the queue length that a customer arriving at time $t$ faces, then (e.g., \cite{JS2013}),
\begin{equation}\label{eq:queue_exp1}
q(t;F):=\E[Q_F(t)]=\lambda F(t-)+\frac{1}{2}\lambda (F(t)-F(t-))-\mu\int_0^t\P(Q_F(u)>0)\diff u \ ,
\end{equation}
where $g(t-):=\lim_{s\uparrow t}g(s)$ for any function $g$. The first term is the number of expected arrivals up to time $t$, the second term is the proportion of the new additional customers that also arrived at time $t$ and are admitted before the tagged customer (recall that customers arriving together are ordered randomly), and the third term corresponds to the expected number of departures. By the memoryless property of the exponential service times, the expected waiting time of a tagged customer arriving at $t$ is
\[
w(t;F)=\frac{1}{\mu}q(t;F)\ .
\]
Minimizing expected waiting time is therefore equivalent to minimizing expected queue length. Definition \ref{def:NE} of the Nash equilibrium states that an arrival strategy $F$ is a symmetric equilibrium if and only if the expected delay is constant on all of the support $F$ and higher outside of the support $\sigma(F)$. In other words, if $F$ is an equilibrium arrival distribution then there exists some constant $C>0$ such that $q(t;F)=C$ for any time $t$ that is chosen with positive probability. We next outline how such a solution can be constructed.

Clearly, $F$ has an atom at time $t=0$; $F(0)>0$. Otherwise, $q(0;F)=0$ and arriving at $t=0$ is the best response. Given $F(0)$ the expected queue length at time zero is
\begin{equation}\label{eq:q0_exp1}
q_0:=q(0;F)=\frac{\lambda F(0)}{2}\ .
\end{equation}
Note that here we have used the assumption that customers arriving together are randomly ordered and therefore half of those arriving at zero will be in front of the tagged customer if he joins at time zero as well. By Definition \ref{def:NE}, in equilibrium $q(t;F)=q(0;F)$ for any $t\in\sigma(F)$. After the opening at $t=0$ there is an interval $(0,t_e)$ with no arrivals, i.e., $F(t_e)=F(0)$, because \eqref{eq:queue_exp1} has an upward discontinuity at zero. Intuitively, arriving immediately after a mass of arrivals cannot be a best response. Moreover, there cannot be an additional atom in any $t\in(0,T]$ because arriving just before the atom yields a lower expected queue length; $q(t-;F)<q(t;F)$. Therefore, the remainder of the distribution is continuous, with density $f$, on the interval $[t_e,T]$, such that $q(t;F)=q(0;F)$ for all $t\in[t_e,T]$ and $F(0)+\int_{t_e}^Tf(t)\diff t=1$. The equilibrium density is given by the unique solution to the functional differential equation (obtained by taking derivative of \eqref{eq:queue_exp1}),
\begin{equation}\label{eq:f1}
\lambda f(t)=\mu\P(Q_F(t)>0)\ .
\end{equation}
There is no explicit solution to \eqref{eq:f1} because the transient probability $\P(Q_F(t)>0)$ does not admit a closed form solution due to the elaborate dependence on the arrival strategy $F$. Nevertheless, the solution is unique and can be computed by solving the Kolmogorov backward equations of the Markovian queue length process. We refer interested readers to \cite{GH1983,HK2011,JS2013} for more details on this model. To summarize, the equilibrium arrival distribution is given by an atom at zero $F(0)$ and a continuous arrival density $f(t)$ on $[t_e,T]$ for some $t_e\in[0,T]$. Note that $F(0)=1$ and $t_e=T$ are possible if $\lambda$ is very big relatively to $\mu T$.

\subsection*{Heterogeneous beliefs game}

Moving on to the case of customers with heterogeneous beliefs regarding the service rate. The beliefs of customers of type $i\in\mathcal{C}$ corresponds to an exponential random variable with rate $\mu_i$ such that $\mu_b>\mu_a$. A symmetric equilibrium is now given by a pair of arrival distributions $(F_a,F_b)$ simultaneously satisfying the equilibrium conditions. Each service time belief corresponds to a different queueing process. In particular, let $Q_{i}$ denote the queue length process corresponding to belief $i\in\mathcal{C}$, while keeping in mind that this process is a function of the arrival distribution $(F_a,F_b)$, then the expected queue length of \eqref{eq:queue_exp1} is adapted to 
\begin{equation}\label{eq:queue_exp}
q_i(t):=\E[Q_i(t)] = \sum_{j\in\mathcal{C}}\lambda_j F_j(t-)+\frac{1}{2}\sum_{j\in\mathcal{C}}\lambda_j(F_j(t)-F_j(t-))-\mu_i\int_0^t\P(Q_{i}(u)>0)\diff u \ .
\end{equation}
As before, the equilibrium condition for a type $i$ customer is that there exists some constant $C_i>0$ such that $q_i(t)=C_i$ for all $t\in\sigma(F_i)$. However, we now seek a pair $(F_a,F_b)$ that simultaneously solves the condition for both types. First observe that the expected queue length at time $t=0$ is identical for both types because service has not started yet,
\begin{equation}\label{eq:q0_exp}
q_0:=q_a(0)=q_b(0)=\frac{\lambda_a F_a(0)+\lambda_b F_b(0)}{2}\ .
\end{equation}
For the same reasons as before, the equilibrium solutions have the following properties: there is a positive atom at time $t=0$ for at least one of the equilibrium arrival distributions, i.e., $F_a(0)=F_b(0)=0$ is not possible, arriving just after $t=0$ is not optimal for both type of customers and there are no additional atoms during $t\in(0,T]$. Hence, for each type $i$ with $F_{i}(0) > 0$ there exists a time $t_i\in(0,T)$ such that $F_i(t_i)=F_i(0)$ and $F_i$ is continuous on $(t_i,T]$. The first important difference from the homogeneous case is that $F_i(0)=0$ for $i\in\mathcal{C}$ implies $q_i(t)\leq q_0$ for any $t\in\sigma(F_i)$. In other words it is possible that one type of customers arrive continuously during the service period with no mass at the opening (but this is not possible for both types simultaneously).
 
The dynamic equations for the continuous part of the of the equilibrium distributions are as follows. If only the arrival density of type $a$ customers is positive at time $t\in(0,T]$; $f_a(t)>0=f_b(t)$ then taking derivative of \eqref{eq:queue_exp} yields 
\begin{equation}\label{eq:f_a_only}
\lambda_a f_a(t)=\mu_a\P(Q_a(t)>0)\ .
\end{equation}
Similarly, if only the arrival density of type $b$ customers is positive at time $t\in(0,T]$; $f_b(t)>0=f_a(0)$, then
\begin{equation}\label{eq:f_b_only}
\lambda_b f_b(t)=\mu_b\P(Q_b(t)>0)\ ,
\end{equation}
However, when both types of customers arrive simultaneously; $f_a(t),f_b(t)>0$, then the equilibrium condition is given by
\begin{equation}\label{eq:f_i_only}
\begin{split}
\lambda_a f_a(t)+\lambda_b f_b(t) &= \mu_a\P(Q_a(t)>0)\ , \\
\lambda_a f_a(t)+\lambda_b f_b(t) &= \mu_b\P(Q_b(t)>0)\ .
\end{split}
\end{equation}
Observe that \eqref{eq:f_i_only} in fact generalizes \eqref{eq:f_a_only} and \eqref{eq:f_b_only} by setting one of the densities to zero in either case. The Markov chain corresponding to the queue length process $Q_i(t)$ has the same nonhomogeneous arrival rate $\lambda_a f_a(t)+\lambda_b f_b(t)$ for both $i=a,b$, but the output rate $\mu_i$ is different for each type. The transient idle probabilities can as before be numerically evaluated by solving the Kolmogorov backward equations.

We next state two lemmas that will be used to characterize the equilibrium arrival profile. The proof of Lemma \ref{lemma:queue_order} relies on a coupling argument for the virtual workload which is detailed in Appendix \ref{sec:appn_A}. 
\begin{lemma}\label{lemma:queue_order}
For any arrival profile $(F_a,F_b)$ such that $F_i(0)>0$ for at least one of $i\in\mathcal{C}$ the expected queue length faced by type $a$ customers is higher than that faced by type $b$ customers: $q_a(t) > q_b(t)$ for all $t>0$.
\end{lemma}

\begin{lemma}\label{lemma:F_order}
For any Nash equilibrium $(F_a,F_b)$, if $F_b(0)>0$ then $F_a(0)=1$.
\end{lemma}
\begin{proof}	
If $F_b(0)>0$ then the equilibrium conditions state that $q_b(t)\geq q_0$ for all $t>0$, which implies that $q_a(t) > q_0$ since $q_a(t) > q_b(t)$ from Lemma \ref{lemma:queue_order}. Therefore, $F_a(dt)>0$ (i.e., $t\in\sigma(F_a)$) for some $t > 0$ contradicts the equilibrium assumption because $q_a(t) > q_0$, and we conclude that $F_a(0) = 1$.
\end{proof}
The previous two lemmas leave us with two following possibilities for Nash equilibrium arrival profiles that are summarized in the following theorem.
\begin{theorem}\label{thm:NE}
If $(F_a,F_b)$ are Nash equilibrium arrival strategies then only one of the following holds:
\begin{enumerate}
\item[(i)] $F_a(0)=1$, $F_b(0)>0$ and $\int_{t_b}^T f_b(t)\diff t=1-F_b(t)$, where $f_b(t)$ is the unique solution to \eqref{eq:f_b_only}.
\item[(ii)] $F_a(0)\in(0,1]$, $F_b(0)=0$ and there exist two arrival sets $\mathcal{T}_a,\mathcal{T}_b\subseteq(0,T]$ such that $\int_{t\in\mathcal{T}_a} f_a(t)\diff t=1-F_a(0)$ and $\int_{t\in\mathcal{T}_b} f_b(t)\diff t=1$, where $(f_a(t),f_b(t))$ satisfy \eqref{eq:f_a_only}, \eqref{eq:f_b_only} and \eqref{eq:f_i_only}. Note that $\mathcal{T}_a=\emptyset$ if $F_a(0)=1$.
\end{enumerate}
\end{theorem}
We conjecture that the second type of equilibrium in fact has a much more specific form, namely that excluding $t=0$ customers of different types arrive on disjoint intervals. This conclusion relies on the following conjecture regarding the equilibrium dynamics in the interior of the acceptance period.
\begin{conjecture}\label{conj:disjoint_intervals}
There is no time $t>0$ such that both types of customers arrive simultaneously; $\sigma(F_a)\cap\sigma(F_b)\cap(0,T]=\emptyset$.
\end{conjecture}
The reasoning for this conjecture is as follows: If type $i\in\cal C$ customers arrive at a positive rate, $\lambda_i f_i(t)>0$, then the functional differential equation \eqref{eq:f_i_only} maintaining the expected queue length $q_i(t)$ is constant must be satisfied. Suppose that on an interval $(t_1,t_2)$ at least one of the types $i\in\cal C$ arrives with a positive density, i.e., $f_i(t)>0$ for all $t\in(t_1,t_2)$. Let $\Lambda_i(t):=\lambda_a f_a(t)+\lambda_b f_b(t)$ denote a solution of the functional differential equations \eqref{eq:f_i_only}, one equation for each type $i\in\cal C$, then by Lemma 9 of \cite{JS2013} each equation admits a unique solution $\Lambda_i()$ on the interval $(t_1,t_2)$. Furthermore, these unique solutions are monotone increasing with the initial queue length distribution $Q_i(t_1)$. We have been able to rule out the existence of a pair of distinct initial distributions $(Q_a(t_1),Q_b(t_1))$ along with rates $\mu_a<\mu_b$ such that the two solutions for two different differential equations coincide on a non-empty interval; $\Lambda_a(t)=\Lambda_b(t)$ for all $t\in[t_1,t_2]$. Such a situation seems unlikely but due to the elaborate dynamics of the functional differential equations we have been unable to establish a proof of non-existence.

If Conjecture \ref{conj:disjoint_intervals} does not hold then there exist multiple equilibria because any combination of $(f_a(t),f_b(t))$ such that $\Lambda_a(t) = \Lambda_b(t)=\lambda_a f_a(t)+\lambda_b f_b(t)$ is also an equilibrium. Otherwise, the uniqueness arguments for the single-class case (e.g., see \cite{JS2013}) can be applied directly to the equilibrium arrival rates on the disjoint arrival intervals.

\section{Fluid approximation}\label{sec:fluid}

We next consider a fluid model that approximates the case where the expected population size is large and service times are very small. In the exponential setting one can think of the limiting system: $\lambda_i\to\infty$ and $\mu_i\to\infty$ as $n\to\infty$ for $i\in\mathcal{C}$, while maintaining $\frac{\lambda_i}{\mu_i}\to K_i$ for some constant $K_i\in(0,\infty)$. Note that this approximation is applicable also to non-Markovian systems for which the service times are not exponentially distributed and/or the total number of customers is not Poisson distributed. Fluid models are often used to analyze queueing systems with elaborate dynamics (e.g., \cite{CM1991}). In particular, they are also used in many queueing game papers, e.g., \cite{EM2016} for a join/balk game in an observable system with a server that alternates between fast and slow service. We refer the reader to \cite{JS2013} for rigorous justification of this approximation for the bottleneck queue arrival game (see also \cite{H2013}).

There are two deterministic fluid populations of volumes $\lambda_a$ and $\lambda_b$ with each individual customer corresponding to a drop with infinitesimal service duration. Type $i\in\mathcal{C}$ customers believe that the deterministic output rate is $\mu_i$, where $\mu_a<\mu_b$. As before, we denote the strategies of type $i$ customers by the cdf $F_i$. Given strategies $(F_a,F_b)$ the fluid input to the queue during any interval $(t_1,t_2)\subset \mathcal{T}$ is 
\[
\lambda_a(F_a(t_2)-F_a(t_1))+\lambda_b(F_b(t_2)-F_b(t_1))\ .
\] 
The equilibrium solutions for $F_i$ is continuous, for reasons detailed below, with a density denoted by $f_i$. The potential output corresponding to the belief of type $i$ customers during the interval is $\mu_i(t_2-t_1)$, where the actual output may be lower in case the queue is empty during some or all of the interval. In particular, if we denote the queue length at time $t$ corresponding to belief $i\in\mathcal{C}$ by $q_i(t)$, then the output rate of the system at $t$ is 
\[
\mu_i\mathbf{1}(q_i(t)>0)+\min\{\mu_i,\lambda_a f_a(t)+\lambda_b f_b(t)\}\mathbf{1}(q_i(t)=0)\ .
\]   
As we will show, there are equilibria such that the queue is empty for any $t\in[0,T]$ for one or both types of customers. We refer to such cases as a degenerate equilibrium because they imply that customers experience no delays and the system is in fact not a bottleneck. The deterministic dynamics of the queue at $t\in[0,T]$ as long as the queue is never empty, which is the case in a non-degenerate equilibrium, are given by
\begin{equation}\label{eq:fluid_qt}
q_i(t)= \sum_{j\in\mathcal{C}} \lambda_j F_j(t)  -\mu_i t, \quad i\in\mathcal{C}\ .
\end{equation} 
Relying on Definition \ref{def:NE}, a Nash equilibrium is a profile $(F_a,F_b)$ such that for $q_i(t)$ is constant throughout the support $\sigma(F_i)$ (and higher or equal outside of the support). Clearly, if there is a discontinuity in any arrival distribution, $F_i(dt)>0$ for some $i=1,2$, then $q_i(t)>q_i(t-)$, which is not possible in equilibrium. Therefore, $F_i$ are indeed continuous in equilibrium. 

Due to the deterministic dynamics of \eqref{eq:fluid_qt} the queue length process for type $a$ dominates that of type $b$, as was established for the stochastic model in the previous section, and therefore we obtain the same structure as in Theorem \ref{thm:NE}. In particular, if a positive volume of optimistic customers arrive at the opening, $F_b(0)>0$, then all pessimistic customers arrive at the opening, $F_a(0)=1$.  If, however, $F_b(0)=0$ then $F_a(0)>0$ and the remaining customers of each type arrive continuously on disjoint intervals (as conjectured in Conjecture \ref{conj:disjoint_intervals}). Furthermore, in the latter case there are parameter settings such that the optimistic customers expect zero delay while the pessimistic customers expect a positive delay. Moreover, there may exist multiple equilibria of this type. For any non-degenerate equilibrium, i.e., one that all customers expect positive delay, all of the pessimistic customers arrive before the optimistic customers. Finally, there also exist multiple fully degenerate equilibria such that all customers expect zero delay and a queue is never formed. The degenerate solutions arise in cases such that the system capacity, corresponding to one or both types of beliefs, is high enough (with respect to the length of the admission period) to process all incoming.

We first state a lemma that establishes the properties of the equilibrium discussed above. These properties correspond to Lemma \ref{lemma:F_order} and Conjecture \ref{conj:disjoint_intervals} in the stochastic model of Section \ref{sec:exp}. The main result of this section is Theorem \ref{thm:NE_fluid} that details the explicit solution and conditions for all cases.

\begin{lemma}\label{lemma:F_fluid}
For any non-degenerate Nash equilibrium, i.e., $q_i(t)>0$ for any $i\in\mathcal{C}$ and $t\in[0,T]$, the following hold:
\begin{enumerate}
\item[(i)] If $F_b(t)>0$ for $t\in[0,T]$ then $F_a(t)=1$.
\item[(ii)] There is no time $t>0$ such that both types of customers arrive simultaneously; $\sigma(F_a)\cap\sigma(F_b)\cap (0,T] = \emptyset$.
\end{enumerate}
\end{lemma}
\begin{proof}
We first establish (i) for $t=0$. The expected queue size ahead of an arbitrary chosen customer arriving at time $t = 0$ is $q_0=\frac{\lambda_a F_a(0)+\lambda_b F_b(0)}{2}$, regardless of the customer type, and the respective waiting time for a type $i$ customer is $\frac{q_0}{\mu_i}$. If $F_b(0)>0$ then $q_{b}(t) \ge q_{0}$ for all $t > 0$ by Definition~\ref{def:NE}. From the deterministic dynamics of the queue in \eqref{eq:fluid_qt}, $\mu_b>\mu_a$ implies that $q_a(t)>q_b(t)\geq q_0$ for all $t > 0$, hence $F_a(0)=1$

Next we verify (ii). Let $f_a,f_b$ denote the density of the continuous equilibrium arrival of types $a$ and $b$, respectively. If $f_b(t)>0$ for $t\in(0,T]$ then the equilibrium condition of keeping the expected queue length $q_b(t)$ constant is obtained by taking derivative of \eqref{eq:fluid_qt}, yielding
\[
 \lambda_a f_a(t)+\lambda_b f_b(t)=\mu_b >\mu_a \quad \text{ for } t \in \sigma(F_{b}) \cap (0,T].
\]
Therefore, the expected queue length $q_a(t)$ is increasing in $t \in \sigma(F_{b}) \cap (0,T]$ and and thus $f_a(t)=0$, which implies that $\sigma(F_a)\cap\sigma(F_b)\cap (0,T] =\emptyset$.

We are left with verifying (i) for $t\in(0,T]$. 
We consider a non-degenerate equilibrium $(F_a,F_b)$ such that $F_a(0)\in(0,1)$ and $F_b(0)=0$. In what follows, we show that if there exists some $t$ such that $F_a(t)<1$ and $F_b(t)>0$ then $q_b(t)=0$ for any $t\in\sigma(F_b)$, which contradicts the definition of ``non-degenerate''
Let $t_b>0$ denote the first time with a positive arrival rate of type $b$ customers; i.e., $f_b(t_b)>0$. If $F_a(t_b)<1$ then there exists some $t_a>t_b$ such that type $a$ customers arrive with a positive rate at time $t_a$, i.e., $f_a(t_a)>0$. The queue length at $t\in(0,t_b)$ for type $b$ customers is given by equation \eqref{eq:fluid_qt},
\[
q_b(t)=2q_0-\mu_b t =\lambda_a F_a(t)-\mu_b t \ ,
\]
and for $t\in[t_b,t_a)$,
\[
q_b(t)=\lambda_a F_a(t_b)+\lambda_b\int_{t_b}^t f_b(u)\diff u-\mu_b t \ .
\]
If $q_b(t_b)>0$ then in equilibrium we have that for $t\in[t_b,t_a)$, $q_b(t)=q_b(t_b)$ and
\[
\frac{\diff}{\diff t} q_b(t)=\lambda_b f_b(t)-\mu_b  =0 \ .
\]
Therefore, for $t\in[t_b,t_a)$ we further have that
\[
q_a(t)=\lambda_a F_a(t_b)+\lambda_b\int_{t_b}^t f_b(u)\diff u-\mu_a t \ ,
\]
and, as $\mu_a<\mu_b$,
\begin{align}
\label{eqn:diff_q_a(t)>0} 
	\frac{\diff}{\diff t} =\lambda_b f_b(t)-\mu_a >\lambda_b f_b(t)-\mu_b =0\ .
\end{align}
Recall that $0 \in \sigma(F_{a})$, $t_{b} \notin \sigma(F_{a})$ and Definition~\ref{def:NE} imply that $q_0 = q_a(0) \le q_a(t_b)$, hence from (\ref{eqn:diff_q_a(t)>0}), we have
\[
q_0 \le q_a(t_b) < q_a(t_a),
\]
which contradicts the equilibrium assumption. We conclude that the only possible equilibrium such that $t_b<t_a$ is if
\[
\lambda_b f_b(t) \leq \mu_a <\mu_b \quad \forall t\in\sigma(F_b) \ ,
\]
and therefore $q_b(t)=0$ for all $t\in\sigma(F_b)$.
\end{proof}

\begin{theorem}\label{thm:NE_fluid}
The (possibly multiple) Nash equilibria $(F_a,F_b)$ satisfy the following:
\begin{enumerate}
\item[(i)] If $T\leq\frac{\lambda_a+\lambda_b}{2\mu_b}$ then $F_a(0)=F_b(0)=1$.
\item[(ii)] If $\frac{\lambda_a+\lambda_b}{2\mu_b}<T<\frac{\lambda_a+2\lambda_b}{2\mu_b}$ then $F_a(0)=1$, $F_b(0)\in(0,1)$ and $f_b(t)=\frac{\mu_b}{\lambda_b}$ for $t\in(t_b,T]$, where $t_b=\frac{\lambda_a+\lambda_b F_b(0)}{2\mu_b}$ and $F_b(0)=\frac{2\mu_b}{\lambda_b}\left(\frac{\lambda_a+2\lambda_b}{2\mu_b}-T\right)$.
\item[(iii)] If $\frac{\lambda_a+2\lambda_b}{2\mu_b}\leq T\leq \frac{\lambda_a}{2\mu_a}+\frac{\lambda_b}{\mu_b}$ then $F_a(0)=1$, $F_b(0)=0$ and $f_b(t)=\frac{\mu_b}{\lambda_b}$ for $t\in(t_b,T]$, where $t_b=T-\frac{\lambda_b}{\mu_b}$.
\item[(iv)] If $\frac{\lambda_a}{2\mu_a}+\frac{\lambda_b}{\mu_b}<T \leq \frac{\lambda_a}{\mu_a}+\frac{\lambda_b}{\mu_b}$, then for type $a$ customers, $F_a(0)\in(0,1)$, $f_a(t)=\frac{\mu_a}{\lambda_a}$ for $t\in[t_a,t_b)$ where $t_a=\frac{\lambda_aF_a(0)}{2\mu_a}$, $t_b=T-\frac{\lambda_b}{\mu_b}$ and $F_a(0)=\frac{2\mu_a}{\lambda_a}\left(\frac{\lambda_a}{\mu_a}+\frac{\lambda_b}{\mu_b}-T\right)$. For type $b$, $F_b(0)=0$, $f_b(t)=\frac{\mu_b}{\lambda_b}$ for $t\in[t_b,T]$.
\item[(v)] Suppose one of the following conditions holds:
\begin{enumerate}
\item[(v.a)] $\mu_b\geq 2\mu_a$ and
\[
\frac{\lambda_a+2\lambda_b}{2\mu_a}<T<\frac{\lambda_a+\lambda_b}{\mu_a}\ .
\]
\item[(v.b)] $\mu_b<2\mu_a$ and
\[
\max\left\lbrace \frac{\lambda_a+2\lambda_b}{2\mu_a},\frac{\lambda_a+\lambda_b}{\mu_a}-\frac{\lambda_b\mu_b}{\mu_a(2\mu_a-\mu_b)}\right\rbrace<T<\frac{\lambda_a+\lambda_b}{\mu_a}\ .
\]
\end{enumerate}
Then there exists at least one equilibrium such that the equilibrium strategy of type $a$ customers is given by $F_a(0)=2\frac{\lambda_a+\lambda_b-\mu_a T}{\lambda_a} \in (0,1)$, $f_a(t)=\frac{\mu_a}{\lambda_a}$ for $t\in[t_a,T]$ where $t_a=\frac{\lambda_a F_a(0)+2\lambda_b}{2\mu_a}$. The strategy of type $b$ customers is given by $F_b(0)=0$, $f_b(t)=\frac{\mu_a\mu_b}{(\lambda_a+\lambda_b-\mu_a T)(\mu_b-2\mu_a)+\lambda_b\mu_b}$ for $t\in[t_b,t_a]$, where $t_b=\frac{\lambda_a F_a(0)}{\mu_b}$.
\item[(vi)] If $T>\frac{\lambda_a}{\mu_a}+\frac{\lambda_b}{\mu_b}$ then there are multiple equilibria such that there are no arrivals at zero, $F_a(0)=F_b(0)=0$, both types can arrive on continuously and the queue is always empty.
\end{enumerate}
\end{theorem}
\begin{remark}
The equilibrium solutions of cases (i)-(iii) are unique. Observe, however, that the intervals for $T$ in (iv) and (v) are not necessarily disjoint, and therefore, if they intersect then both cases are equilibrium solutions. Similarly, the intervals may also intersect for cases (v) and (vi). Case (v) is a partially degenerate equilibrium because type $b$ customers expect zero delay and case (vi) is fully degenerate because all customers expect zero delay. In the degenerate cases of (v) and (vi) there exist a continuum of equilibria solutions because the arrival densities need not be uniform as long as the queue always remains empty.
\end{remark}
\begin{proof}
The proof constructs the possible equilibrium solutions and verifies the conditions are as given by the theorem statement. The first three cases correspond to the unique non-degenerate solutions satisfying Lemma \ref{lemma:F_fluid}(ii). In the last two cases we construct degenerate equilibrium such that the queue is always empty for at least one type of customers. In the latter cases we further explain why there exist multiple equilibria solutions. Case (iv) is a bit more involved because, depending on the specific parameters, both degenerate and non-degenerate equilibria are possible.
\begin{enumerate}
\item[(i)] If $F_a(0)=F_b(0)=1$ then the expected queue size at $t>0$ for a customer of type $i$ is
\[
q_i(u)=2q_0-\mu_i u=\lambda_a+\lambda_b-\mu_i u, \quad u\in(0,t)\ .
\]
A type $b$ customer will arrive at $t$ if $q_b(t)\leq q_0$, i.e., $t\geq t_b=\frac{\lambda_a+\lambda_b}{2\mu_b}$, but $T\leq \frac{\lambda_a+\lambda_b}{2\mu_b} \leq t_b$ and so no type $b$ customer will deviate from $F_b$. Furthermore, $\mu_a<\mu_b$ implies that $q_a(t)\geq q_b(t)$ and so customers of type will certainly not deviate from strategy $F_a$, hence this is indeed an equilibrium and case (i) is verified.
\item[(ii)] Next consider $F_a(0)= 1$ and $F_b(0)\in[0,1)$. In this case $q_0=\frac{\lambda_a+\lambda_b F_b(0)}{2}$ and for any type $b$ customer the expected queue size at $t>0$ such that there are no arrivals during $(0,t)$ is
\[
q_b(u)=2q_0-\mu_b u=\lambda_a+\lambda_b F_b(0)-\mu_b u, \quad u\in(0,t)\ .
\]
Solving $q_0=q_b(t_b)$ yields that $f_b(t)=0$ for $t<t_b=\frac{\lambda_a+\lambda_b F_b(0)}{2\mu_b}$. By applying \eqref{eq:fluid_qt}, the equilibrium condition $q_b(t)=q_0$ for $t\in[t_b,T]$ yields $f_b(t)=\frac{\mu_b}{\lambda_b}$. The probability $F_b(0)$ is obtained by solving $F_b(0)+(T-t_b)\frac{\mu_b}{\lambda_b}=1$, yielding $F_b(0)=\frac{2\mu_b}{\lambda_b}\left(\frac{\lambda_a+2\lambda_b}{2\mu_b}-T\right)$. Finally $F_b(0)<1$ is equivalent to $T>\frac{\lambda_a+\lambda_b}{2\mu_b}$ and $F_b(0)\geq 0$ is equivalent to $T\leq \frac{\lambda_a+2\lambda_b}{2\mu_b}$, yielding case (ii).
\item[(iii)] Next consider $F_a(0)=1$ and $F_b(0)=0$. In this case all type $b$ customers arrive uniformly with density $f_b(t)=\frac{\mu_b}{\lambda_b}$ for $t\in(t_b,T]$, where $t_b=T-\frac{\lambda_b}{\mu_b}$. Type $a$ customers will not deviate from $F_a$ if 
\[
q_a(t_b)=\lambda_a-\mu_a t_b=\lambda_a-\mu_a\left(T-\frac{\lambda_b}{\mu_b}\right)\geq q_0=\frac{\lambda_a}{2}\ ,
\]
and this is equivalent to $T\leq \frac{\lambda_a}{2\mu_a}+\frac{\lambda_b}{\mu_b}$. Moreover, type $b$ customers will not deviate if 
\[
q_b(t_b)=\lambda_a-\mu_b t_b=\lambda_a-\mu_b\left(T-\frac{\lambda_b}{\mu_b}\right)\leq q_0=\frac{\lambda_a}{2}\ ,
\]
which is equivalent to $T>\frac{\lambda_a+2\lambda_b}{2\mu_b}$ and case (iii) is confirmed.
\item[(iv)] Now assume that $F_a(0)\in[0,1)$, $F_b(0)=0$, $f_a(t)=\frac{\mu_a}{\lambda_a}$ for $t\in[t_a,t_b)$ where $t_a=\frac{\lambda_aF_a(0)}{2\mu_a}$, $F_b(0)=0$ and $f_b(t)=\frac{\mu_b}{\lambda_b}$ for $t\in[t_b,T]$. Solving $F_a(0)+(t_b-t_a)\frac{\mu_a}{\lambda_a}=1$ and $(T-t_b)\frac{\mu_b}{\lambda_b}=1$ yields $F_a(0)=\frac{2\mu_a}{\lambda_a}\left(\frac{\lambda_a}{\mu_a}+\frac{\lambda_b}{\mu_b}-T\right)$ and $t_b=T-\frac{\lambda_b}{\mu_b}$. Furthermore, $F_a(0)<1$ is equivalent to $T> \frac{\lambda_a}{2\mu_a}+\frac{\lambda_b}{\mu_b}$ and $F_a(0)\geq 0$ is equivalent to $T\leq \frac{\lambda_a}{\mu_a}+\frac{\lambda_b}{\mu_b}$, verifying case (iv).
\item[(v)] We next consider a solution such that $F_a(0)\in(0,1)$ and $F_b(0)=0$ such that $F_a(t_a)<F_b(t_a)=1$ for some $t_a<T$. Let $t_b$ denote the first time such that $f(t_b)>0$. Lemma \ref{lemma:F_fluid} implies that $q_b(t_b)=0$ in equilibrium, hence,
\begin{equation}\label{eq:cond1}
\lambda_a F_a(0)-\mu_b t_b=0 \ \Leftrightarrow \ t_b=\frac{\lambda_a F_a(0)}{\mu_b}\ .
\end{equation}
For $t\in[t_b,t_a)$, the type $b$ density is $f_b(t)=k$, for some $k>0$ such that $\lambda_b k\leq\mu_a<\mu_b$. The latter condition is necessary because Definition \ref{def:NE} states that $q_a(t_a)=q_0$ and $q_a(t_b)\geq q_0$ (as $t_a\in\sigma(F_a)$ and $t_b\notin\sigma(F_a)$), hence $\lambda_b k\leq \mu_a$ ensures that $q_a(t_a)\leq q_a(t_b)$. In addition, $(t_a-t_b)k=1$ yields
\begin{equation}\label{eq:cond2}
f_b(t)=k=(t_a-t_b)^{-1},\ t\in[t_b,t_a)\ .
\end{equation}
Furthermore, in equilibrium we have that 
\[
q_a(t_a)=\lambda_a F_a(0)+\lambda_b-\mu_a t_a=q_0=\frac{\lambda_a F_a(0)}{2}\ ,
\] 
hence
\begin{equation}\label{eq:cond3}
t_a=\frac{\lambda_a F_a(0)+2\lambda_b}{2\mu_a}\ .
\end{equation}
For $t\in[t_a,T]$ we have that $f_a(t)=\frac{\mu_a}{\lambda_a}$, yielding
\begin{equation}\label{eq:cond4}
(T-t_a)\frac{\mu_a}{\lambda_a}+F_a(0)=1\ .
\end{equation}
Combing \eqref{eq:cond3} and \eqref{eq:cond4} yields
\[
F_a(0)=2\frac{\lambda_a+\lambda_b-\mu_a T}{\lambda_a}\ .
\]
Combining \eqref{eq:cond1}-\eqref{eq:cond4} yields
\[
k=\frac{\mu_a\mu_b}{(\lambda_a+\lambda_b-\mu_a T)(\mu_b-2\mu_a)+\lambda_b\mu_b}\ .
\]
Recall that for this to be an equilibrium there is an extra condition $0<\lambda_b k\leq\mu_a$. Furthermore, $0<F_a(0)<1$ is equivalent to $\frac{\lambda_a+2\lambda_b}{2\mu_a}<T<\frac{\lambda_a+\lambda_b}{\mu_a}$. First observe that $\lambda_b k \leq \mu_a$ is equivalent to 
\[
T<\frac{\lambda_a+\lambda_b}{\mu_a}\ .
\]
We verify the condition $k>0$ for the three possible cases:
\begin{itemize}
\item If $\mu_b= 2\mu_a$ then $k=\frac{\mu_a}{\lambda_b}$ and then $\lambda_b k=\mu_a$.
\item If $\mu_b> 2\mu_a$ then a necessary condition for $k>0$ is
\[
T<\frac{\lambda_a+\lambda_b}{\mu_a}+\frac{\lambda_b\mu_b}{\mu_a(\mu_b-2\mu_a)}\ .
\]
\item If $\mu_b< 2\mu_a$ then a necessary condition for $k>0$ is
\[
T>\frac{\lambda_a+\lambda_b}{\mu_a}-\frac{\lambda_b\mu_b}{\mu_a(2\mu_a-\mu_b)}\ .
\]
\end{itemize}
To summarize, this equilibrium is possible if one of the conditions (v.a) or (v.b) holds. This equilibrium is not unique because for small $\epsilon$ you can choose $t_b=t_b+\epsilon$ and $k=k+g(\epsilon)$, where $g$ is some continuous function, and all of the conditions will still be satisfied.
\item[(vi)] Finally, we consider the case of $F_a(0)=F_b(0)=0$. We call this a degenerate case because
\[
q_a(t)=q_b(t)=q_0=0,\quad \forall t\in[0,T],
\]
i.e., a queue is never formed and the system is not a bottleneck. Suppose that type $a$ customers arrive uniformly with density $f_a(t)=\frac{\mu_a}{\lambda_a}$ on $t\in[0,T_a]$ and type $b$ customers arrive uniformly with density $f_b(t)=\frac{\mu_b}{\lambda_b}$ on $t\in[T_a,T_b]$. The above define proper distributions if $T_a=\frac{\lambda_a}{\mu_a}$ and $T_b=t_a+\frac{\lambda_b}{\mu_b}$. Therefore, this is an equilibrium if $T\geq T_b=\frac{\lambda_a}{\mu_a}+\frac{\lambda_b}{\mu_b}$, as in case (vi). The equilibrium is not not unique because the above uniform distributions can be constructed on any disjoint subsets of $[0,T]$, or even on subsets with intersection as long as $\lambda_a f_a(t)+\lambda_b f_b(t)\leq \mu_a $ for any $t\in[0,T]$.
\end{enumerate}
\end{proof}

In Figure \ref{fig:fluid} we illustrate all possible non-degenerate equilibrium outcomes for fixed parameters $T=1$, $\lambda_a=1$, $\lambda_b=2$, $\mu_a=1$ and varying values of $\mu_b$. When $\mu_b$ is low then all customers arrive at $t=0$ as seen in Figure \ref{fig:fluid}a. As $\mu_b$ grows the optimistic (type $b$) customers become more optimistic regarding service speed and will spread out more throughout the acceptance period as seen in Figure \ref{fig:fluid}b and Figure \ref{fig:fluid}c. Finally, if $\mu_b$ is very high all of the optimistic customers arrive towards the end of the period which results in pessimistic (type $a$) customers also mixing between $t=0$ and a continuous interval within the acceptance period.

\begin{figure}[H]
\centering
\begin{subfigure}{.45\linewidth}
\centering
\begin{tikzpicture}[xscale=3.3,yscale=2.5]
\def\xmin{0}
  \def\xmax{1.1}
  \def\ymin{0}
  \def\ymax{1.1}
    \draw[->] (\xmin,\ymin) -- (\xmax,\ymin) node[right] {$t$} ;
    \draw[->] (0,\ymin) -- (0,\ymax)  ;
    \foreach \x in {0,1}
    \node at (\x,\ymin) [below] {\x};
    \foreach \y in {0,1}
    \node at (0,\y) [left] {\y};
	
	\draw[red,thick] (0,1)  --  (1,1) node[above] {\textcolor{red}{$F_a(t)$}} ;
	
	\draw[blue,thick,densely dashed] (0,1)  --  (1,1) node[below] {\textcolor{blue}{$F_b(t)$}} ;
\end{tikzpicture}
\caption{$\mu_b=1.5$}
\end{subfigure}
\begin{subfigure}{.45\linewidth}
\centering
\begin{tikzpicture}[xscale=3.3,yscale=2.5]
 \def\xmin{0}
  \def\xmax{1.1}
  \def\ymin{0}
  \def\ymax{1.1}
    \draw[->] (\xmin,\ymin) -- (\xmax,\ymin) node[right] {$t$} ;
    \draw[->] (0,\ymin) -- (0,\ymax)  ;
    \foreach \x in {0,1}
    \node at (\x,\ymin) [below] {\x};
    \foreach \y in {0,1}
    \node at (0,\y) [left] {\y};
	
	\draw[red,thick] (0,1)  --  (1,1) node[above] {\textcolor{red}{$F_a(t)$}} ;
	
	\draw[blue,thick,densely dashed] (0,0.5) -- (0.5,0.5)  --  (1,1) ;
	\node at (1.03,0.7)  {\textcolor{blue}{$F_b(t)$}};
	\node[below] at (0.5,0)  {\textcolor{blue}{$t_b$}};
	\draw[blue,dotted] (0.5,0) -- (0.5,0.5) ;
	\node[left] at (0,0.5)  {\textcolor{blue}{$F_b(0)$}};

\end{tikzpicture}
\caption{$\mu_b=2$}
\end{subfigure}

\begin{subfigure}{.45\linewidth}
\centering
\begin{tikzpicture}[xscale=3.3,yscale=2.5]
 \def\xmin{0}
  \def\xmax{1.1}
  \def\ymin{0}
  \def\ymax{1.1}
    \draw[->] (\xmin,\ymin) -- (\xmax,\ymin) node[right] {$t$} ;
    \draw[->] (0,\ymin) -- (0,\ymax)  ;
    \foreach \x in {0,1}
    \node at (\x,\ymin) [below] {\x};
    \foreach \y in {0,1}
    \node at (0,\y) [left] {\y};
	
	\draw[red,thick] (0,1)  --  (1,1) node[above] {\textcolor{red}{$F_a(t)$}} ;
	
	\draw[blue,thick,densely dashed] (0,0) -- (0.5,0)  --  (1,1) ;
	\node at (1.03,0.7)  {\textcolor{blue}{$F_b(t)$}};
	\node[below] at (0.5,0)  {\textcolor{blue}{$t_b$}};

\end{tikzpicture}
\caption{$\mu_b=4$}
\end{subfigure}
\begin{subfigure}{.45\linewidth}
\centering
\begin{tikzpicture}[xscale=3.3,yscale=2.5]
  \def\xmin{0}
  \def\xmax{1.1}
  \def\ymin{0}
  \def\ymax{1.1}
    \draw[->] (\xmin,\ymin) -- (\xmax,\ymin) node[right] {$t$} ;
    \draw[->] (0,\ymin) -- (0,\ymax) ;
    \foreach \x in {0,1}
    \node at (\x,\ymin) [below] {\x};
    \foreach \y in {0,1}
    \node at (0,\y) [left] {\y};
	
	\draw[red,thick] (0,0.4) -- (0.25,0.4) -- (0.75,1) --  (1,1) node[above] {\textcolor{red}{$F_a(t)$}} ;
	\node[below] at (0.25,0)  {\textcolor{red}{$t_a$}};
	\node[left] at (0,0.4)  {\textcolor{red}{$F_a(0)$}};
	\draw[red,dotted] (0.25,0) -- (0.25,0.4) ;
	\draw[red,dotted] (0.75,0) -- (0.75,1) ;
	
	\draw[blue,thick,densely dashed] (0,0) -- (0.75,0)  --  (1,1) ;
	\node at (1.07,0.7)  {\textcolor{blue}{$F_b(t)$}};
	\node[below] at (0.75,0)  {\textcolor{blue}{$t_b$}};

\end{tikzpicture}
\caption{$\mu_b=8$}
\end{subfigure}
\caption{Equilibrium arrival distributions (solid red for type $a$ and dashed blue for type $b$) for varying service rate $\mu_b$ of the optimistic type $b$ customers. The other parameters are fixed: $T=1,\lambda_a=1,\lambda_b=2,\mu_a=1$.}
\label{fig:fluid}
\end{figure}
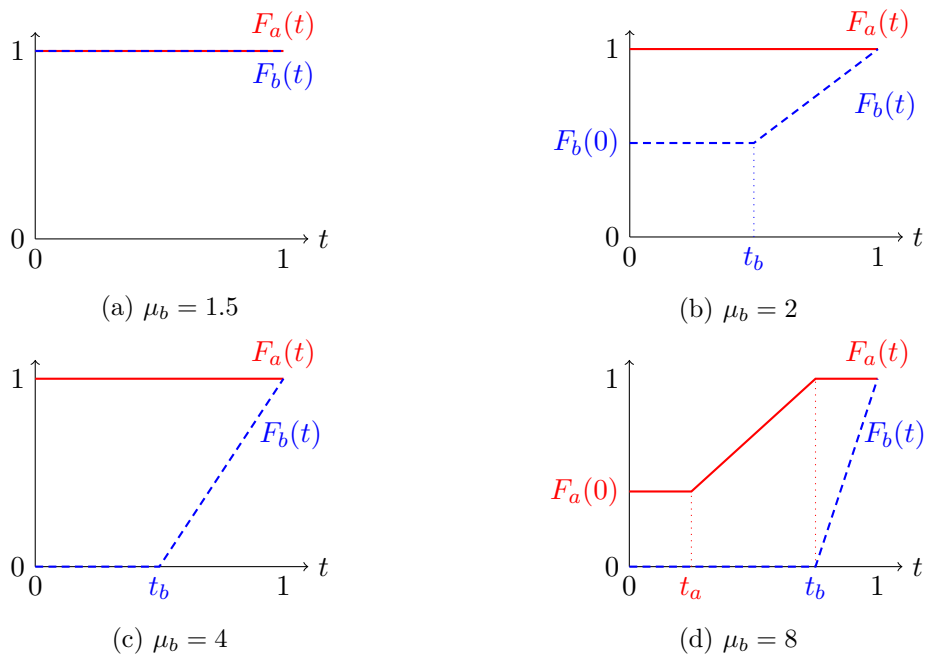

\section{Discrete-time game}\label{sec:discrete}

In this section, we consider a discrete-time queue with an acceptance period and two types of arriving customers with type-dependent beliefs on their service times. There are several motivations for studying the discrete-time system. Most notably, it enables a numerical computation of the equilibrium arrival distribution for general service times (see e.g., \cite{SMF2019} for the case of homogeneous customers with a common belief). The computational advantage is even more relevant for the model with heterogeneous customer beliefs because, as was shown in Section \ref{sec:exp}, the structure of the equilibrium is very elaborate even in the fully Markovian setting with exponential service times. In particular, solving the equilibrium dynamics in continuous time seems intractable and there may be multiple equilibria as indicated by the fluid approximation of Section \ref{sec:fluid}. Another important advantage of the discrete-time system is the flexibility to deal with systems that operate inherently with predetermined time slots for admissions, i.e., in which the acceptance period is divided into a discrete number of slots and arrivals are only allowed in the designated slots (see e.g., \cite{book_BK2012} for a detailed review of the role of discrete time systems in modeling communication systems). Moreover, in \cite{HK2011} it was shown that limiting the number of possible arrival time slots may lower the expected waiting time for all customers in equilibrium. Finally, the discrete-time system can also be used to approximate a continuous-time system by considering a very large number of small slots. 
For this discrete-time system, we propose a best-response based algorithm for computing the equilibrium arrival distributions of both types of customers, and show the existence of the equilibrium arrival distributions (see Sections \ref{sec:algo_NE} and \ref{sec:discrete_NE_existence}). In Section \ref{sec:numerical} we present some numerical examples and show that our algorithm gives similar equilibrium arrival distribution to that of the fluid model of Section \ref{sec:fluid}. 

In what follows, we formally describe the discrete-time system studied in this section. 
The time axis $\bbR$ is divided into time units of length one, and the unfinished workload (if any) in the system decreases by one in each time units. The acceptance period is given by $[0,(T+1)\tau) \subset \bbR$, where $\tau$ and $T$ are positive integers. The acceptance period is split into time slots of length $\tau$, i.e.,
\begin{align*}
\label{eqn:}
 [0,(T+1)\tau) = \cup_{t \in \calT} [t\tau, (t+1)\tau),
\end{align*}
where $\calT := \{ 0,1,2,\ldots, T \}$ and $[t\tau, (t+1)\tau)$ is referred to as slot $t$. Customers can choose any slot $t \in \calT$, and arrivals occur in the beginning of the slots. The population sizes of types $a$ and $b$ customers follow iid Poisson random variables with mean $\lambda_a$ and $\lambda_b$, respectively.  We further assume that the service times are positive integer valued random variables for either of the types. For $i \in \calC=\{a,b\} $, let $\vc{x}_{i} := (x_{i}(k) ; k \ge 1)$ denote the service time distribution with type $i$, where its mean is denoted by $\chi_{i} := \sum_{k \ge 1} k x_{i}(k)$. A mixed strategy is now a discrete probability distribution $({p}_{t} ; t \in \calT)$. A symmetric Nash equilibrium is given by a pair of distributions $\vc{p}_{a} $ and $\vc{p}_{b}$ with respective cdf's $(F_a,F_b)$ satisfying Definition \ref{def:NE}.

We first compute the mean unfinished workload for each customer's type ($a$ and $b$) given an arrival distribution relying on the more detailed analysis of \cite{SMF2019}. This is then used to construct a sufficient and necessary condition so that $(\vc{p}_{a},\vc{p}_{b})$ is Nash equilibrium. Finally, we give an algorithm for computing the equilibrium. 

In the following analysis all random variables and expectations are given for a pair of arrival distributions $(\vc{p}_{a},\vc{p}_{b})$, but we omit this from the notation for the sake of brevity. For $i \in \calC$ and $t \in \calT$, let $V_{i,t-}$ denote the unfinished workload in system immediately before slot $t$ with type $i$. Recall that from Definition~\ref{def:belief}, the mean waiting time corresponding to belief $i \in \calC$ is computed by a workload process with iid service times with distribution $\vc{x}_{i}$. The probability distribution of $V_{i,t-}$ is denoted by $\vc{v}_{i,t} = (v_{i,t}(k) ; k \ge 0)$, i.e., $v_{i,t}(k) = \P( V_{i,t-} = k ), \quad k \ge 0, \ i \in \calC, \ t \in \calT$. Since there is no arriving customer before slot $0$, i.e., we have $v_{i,0}(k) = \1(k=0)$, where $\1(\cdot)$ is the indicator function of an event in the parenthesis. 

The total work arriving in slot $t$ is
\[
H_{i,t} = \sum_{k=1}^{N_{t}} X_{i,k}\ ,
\]
where $N_{t}$ is a Poisson distributed random variable with mean $\lambda_{a} p_{a,t} + \lambda_{b} p_{b,t}$, and $X_{i,k}$'s are iid random variables with distribution $\vc{x}_{i}$. Observe that $H_{i,t}$ follows a compound Poisson distribution, and denote its probability distribution by $\vc{h}_{i,t} := (h_{i,t}(k) ; k \ge 0)$, which is recursively computed (see e.g., \cite{ZL2016}). Figure~\ref{fig:samplepathwl} illustrates a sample path of the unfinished workload in view of type $a$ customers.
Recall that $V_{i,t-}$ (resp. $H_{i,t}$), $i \in \calC$ and $t \in  \calT$, denotes the unfinished (resp. total arrival) workload just before the beginning of (resp. in the beginning of) slot $t$ in view of type $i$.
In Figure~\ref{fig:samplepathwl}, we can see that
\[
V_{a,0-}(\omega) = 0, \ V_{a,1-}(\omega) = 2, \ V_{a,2-}(\omega) = 3, \
H_{a,0}(\omega) = 4, \ H_{a,1}(\omega) = 4.
\]
\begin{figure}[H]
	\centering
	\includegraphics[width=0.6\linewidth]{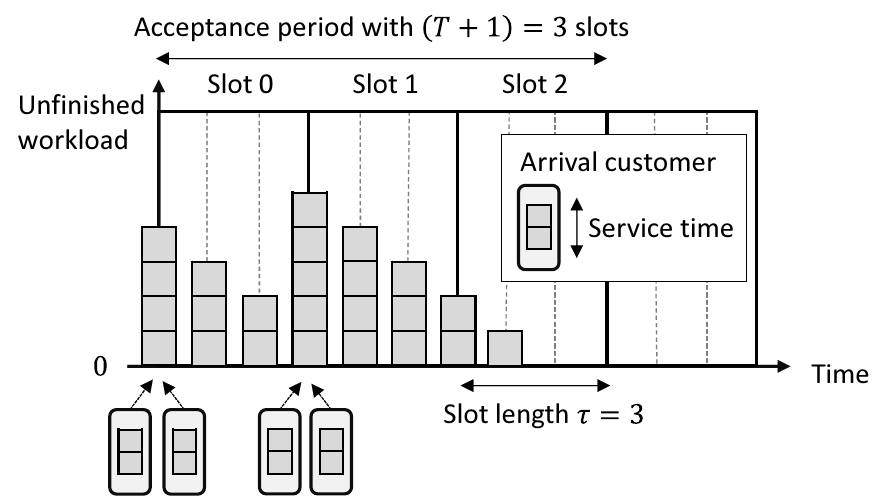}
	\caption{A sample path of the unfinished workload corresponding to a belief of type $a$, where $\tau = 3$, $T = 2$ and with the belief of  deterministic service time of length $2$.}
	\label{fig:samplepathwl}
\end{figure}

The mean unfinished workload with customer's type $i \in \calC$ is computed as follows. 
\begin{lemma} 
For $i \in \calC$ and $t \in \calT$, we have
\begin{align} 
\label{eqn:Mean_unfinished_WL}
\E[V_{i,t-}] = \sum_{u=0}^{t-1} \left\{ (\lambda_{a} p_{a,u} + \lambda_{b} p_{b,u}) \chi_{i} - \tau+ \sum_{k=0}^{\tau-1} (\tau-k) v_{i,u} * h_{i,u} (k) \right\}, 
\end{align}
where $v_{i,u} * h_{i,u} (k) := \sum_{\ell = 0}^{k} v_{i,u}(\ell) h_{i,u}(k-\ell)$, 
and the probability distribution of the unfinished workload $(v_{i,t}(k) ; k \ge 0)$ is recursively calculated by 
\begin{align} 
\label{eqn:PDF_of_v_{i,t}}
v_{i,t}(k) = 
\begin{cases}
 \sum_{\ell = 0}^{\tau} v_{i,t-1} * h_{i,t-1} (\ell), & k = 0\ , \\
 v_{i,t-1} * h_{i,t-1} (\tau + k), & k \ge 1\ .
\end{cases}
\end{align}
\end{lemma}
\begin{proof} 
Since in each slot, the unfinished work (if any) is processed by at most $\tau$ units, i.e.,
\begin{align} 
\label{eqn:Def_of_V_{i,t-}}
V_{i,t-} = (V_{i,(t-1)-} + H_{i,t-1} - \tau)^{+}, \quad i \in \calC, \ t \in \calT\ ,
\end{align}
then 
\begin{align*} 
\E[V_{i,t-}]
&= 
\E[(V_{i,(t-1)-} + H_{i,t-1} - \tau) \1(V_{i,(t-1)-} + H_{i,t-1} \ge \tau)]\\
&= 
\E[V_{i,(t-1)-} + H_{i,t-1} - \tau] + \E[\{ \tau - (V_{i,(t-1)-} + H_{i,t-1}) \} \1(V_{i,(t-1)-} + H_{i,t-1} \le \tau-1)]\ . 
\end{align*}
Since $V_{i,(t-1)-}$ and $H_{i,t-1}$ are independent, the second term in the last equation is given by
\begin{align*} 
\E[\{ \tau - (V_{i,(t-1)-} + H_{i,t-1}) \} \1(V_{i,(t-1)-} + H_{i,t-1} \le \tau-1)]
= 
\sum_{k=0}^{\tau-1} (\tau-k) v_{i,t-1} * h_{i,t-1} (k)\ . 
\end{align*}
By combining these equations, we obtain \eqref{eqn:Mean_unfinished_WL}.
Equation \eqref{eqn:PDF_of_v_{i,t}} immediately follows by \eqref{eqn:Def_of_V_{i,t-}}.
\end{proof}

By Definition~\ref{def:NE}, the arrival distributions in equilibrium are given as follows.
\begin{lemma} \label{lemma:NE_Arrival_Distributions}
A pair of arrival distributions $(\vc{p}_{a},\vc{p}_{b})$ is a symmetric (within types) Nash equilibrium if and only if there exist positive numbers $\ol{w}_{i}$ ($i \in \calC$) such that
\begin{align}
\label{eqn:NE_Arrival_Distributions} 
p_{i,t} = \frac{1}{\lambda_{i}} 
\left(
\frac{2}{\chi_{i}} (\ol{w}_{i} - \E[V_{i,t-}]) - \lambda_{-i} p_{-i,t} 
\right)^{+}, \quad i \in \calC, \ t \in \calT\ ,
\end{align}
where $-i$ denotes the counterpart of $i$, i.e., if $i=a$, then $-i = b$, and vice versa.
\end{lemma}
\begin{proof} 
For $i\in \calC$ and $t \in \calT$, let $w_{i,t}$ denote the expected waiting time of a tagged type $i$ arrival customer if she/he chooses slot $t$. From Definition~\ref{def:NE}, $(\vc{p}_{a},\vc{p}_{b})$ is an equilibrium if there exist positive numbers $\ol{w}_{i}$ ($i \in \calC$) such that
\begin{align} 
\label{eqn:NE_Eq1}
 w_{i,t} &= \ol{w}_{i}, \quad t \in \sigma(\vc{p}_{i})\ ,\\
\label{eqn:NE_Eq2}
 w_{i,t} &\geq \ol{w}_{i}, \quad t \notin \sigma(\vc{p}_{i})\ .
\end{align}
Since
\begin{align} 
\label{eqn:w_{i,t}}
 w_{i,t} = \E[V_{i,t-}] + \frac{\lambda_{a} p_{a,t} + \lambda_{b} p_{b,t}}{2} \chi_{i},
 \quad i \in \calC, \ t \in \calT\ ,
\end{align}
then equations \eqref{eqn:NE_Eq1} and \eqref{eqn:NE_Eq2} are rewritten by
\begin{equation}\label{eq:NE_p+}
p_{i,t} = \frac{1}{\lambda_{i}} 
\left(
 \frac{2}{\chi_{i}} (\ol{w}_{i} - \E[V_{i,t-}]) - \lambda_{-i} p_{-i,t} 
\right), \quad t \in \sigma(\vc{p}_{i})\ ,
\end{equation}
and 
\[
p_{i,t}=0 \geq \frac{1}{\lambda_{i}} 
\left(
 \frac{2}{\chi_{i}} (\ol{w}_{i} - \E[V_{i,t-}]) - \lambda_{-i} p_{-i,t} 
\right), \quad t \notin \sigma(\vc{p}_{i})\ ,
\]
respectively, which imply \eqref{eqn:NE_Arrival_Distributions}.
\end{proof}
%

\subsection{Algorithm for computing equilibrium distributions}\label{sec:algo_NE}

Lemma \ref{lemma:NE_Arrival_Distributions} provides a recursive relation in the form of Equation \eqref{eqn:NE_Arrival_Distributions} that allows the iterative computation of the probability to arrive in every time slot $t \in \calT$. For every type $i\in\mathcal{C}$ this recursion relies on a given distribution for the other type $j\neq i$. We leverage this structure to derive a best-response type algorithm that fixes the distribution of one type $i$ in each iteration and then computes the distribution of the other type $j$, and then repeating this procedure for the $j$ given the computed distribution for $i$. A point of convergence, if it converges to some point, is clearly a Nash equilibrium, although we do not argue that the algorithm is guaranteed to converge. In practice, this procedure finds equilibrium points very efficiently.

For a fixed $i\in \calC$, assume that the arrival distribution of type $-i$ customers is given by $\vc{p}_{-i}$. Then the best response of type $i$ customers is computed by a discretized (with accuracy parameter $\epsilon>0$) search procedure that we call Algorithm~1, which is similar to the methods applied in \cite{HK2011} and \cite{SMF2019}. Algorithm 1 is detailed in Appendix \ref{sec:appn_B}, and here we treat it as a function $\mathbf{Alg.1}(\vc{p}_{-i}, \vc{\lambda}, \vc{x}_{i}, \epsilon)\mapsto \vc{p}_{i}$ that returns a probability distribution $\vc{p}_{i}$ that satisfies \eqref{eqn:NE_Arrival_Distributions} for any distribution of the other type $\vc{p}_{-i}$, system parameters $\vc{\lambda} := (\lambda_{a},\lambda_{b})$, $\vc{x}_{i}$, and accuracy parameter $\epsilon>0$. 

Note that best-response here refers to all customers of type $i$ using the same distribution, and not to an individual best response for an individual customer of type $i$. This is then used iteratively in Algorithm~2 to find equilibrium arrival distributions, i.e., to find positive numbers $(\ol{w}_{a},\ol{w}_{b})$ and the corresponding arrival distributions $(\vc{p}_{a},\vc{p}_{b})$ jointly satisfying \eqref{eqn:NE_Arrival_Distributions}, which was shown to be equivalent to the Definition \ref{def:NE} of a Nash equilibrium in Lemma \ref{lemma:NE_Arrival_Distributions}. 

We next propose an iterated approximate best-response algorithm to search for a pair of equilibrium arrival distributions, where the accuracy of the approximation depends on the specified accuracy parameters $\epsilon$ and $\delta$. 

\begin{algorithm}[H]\label{algo:NE_BR}
\renewcommand{\thealgorithm}{}
\caption{\textbf{2}: Iterated best response}
\textbf{Input}: $\vc{\lambda}:=(\lambda_{a},\lambda_{b})$, $\vc{x} := (\vc{x}_{a}, \vc{x}_{b})$, $\epsilon > 0$, $\delta>0$ \\
\textbf{Output}: $(\vc{p}_{a}^{e},\vc{p}_{b}^{e})$ (equilibrium)
\begin{algorithmic}
\State init $\vc{p}_{a}^{0} := (1,0,\ldots,0)$, $\vc{p}_{b}^{0} := (1,0,\ldots,0)$
\State init $k := 0$, $\Delta :=\delta$
\While{$\Delta \geq \delta$}
	\State set $\vc{p}_{a}^{k+1} := \mathbf{Alg.1}(\vc{p}_{b}^{k}, \vc{\lambda}, \vc{x}_{a}, \epsilon)$, $\vc{p}_{b}^{k+1} := \mathbf{Alg.1}(\vc{p}_{a}^{k+1}, \vc{\lambda}, \vc{x}_{b}, \epsilon)$
	\State set $\Delta := \max\{||\vc{p}_{a}^{k+1}-\vc{p}_{a}^{k}||, ||\vc{p}_{b}^{k+1}-\vc{p}_{b}^{k}||\}$
\State set $k := k+1$
\EndWhile
\State set $(\vc{p}_{a}^{e},\vc{p}_{b}^{e}) := (\vc{p}_{a}^{k},\vc{p}_{b}^{k})$
\State \textbf{return} {$(\vc{p}_{a}^{e},\vc{p}_{b}^{e})$}
\end{algorithmic}
\end{algorithm}

\begin{remark}\label{rem:BR_convergence}
Note that any point of convergence of Algorithm 2 is an equilibrium because no single customer can deviate and obtain a lower expected waiting time by unilaterally changing the arrival distribution. However, we do not provide a proof of convergence of the algorithm. Best response dynamics are known to converge for games with very specific utility functions, such as submodular games (see \cite{T1979}) and potential games (see \cite{MS1996}), but for a system with elaborate and non-explicit dynamics, as is the case here, convergence guarantees are very hard to prove (or may not hold). Nevertheless, because any convergence point is an equilibrium, such algorithms are very useful in computing equilibrium points even in cases where there is no theoretical guarantee. Indeed, in our numerical testing and examples an equilibrium point was reached for every instance of the arrival game tested. Figure~\ref{fig:2slotsbrexample} illustrates an numerical example for the best response of type $i$ ($i = a,b$) customers against $\vc{p}_{-i}$ in the case of $T+1 = 2$, i.e., the acceptance period consists of two slots.  This figure presents a typical path of the iterated best-response algorithm and how it converges to a pair of equilibrium arrival distributions.
\begin{figure}[h]
	\centering
	\includegraphics[width=0.7\linewidth]{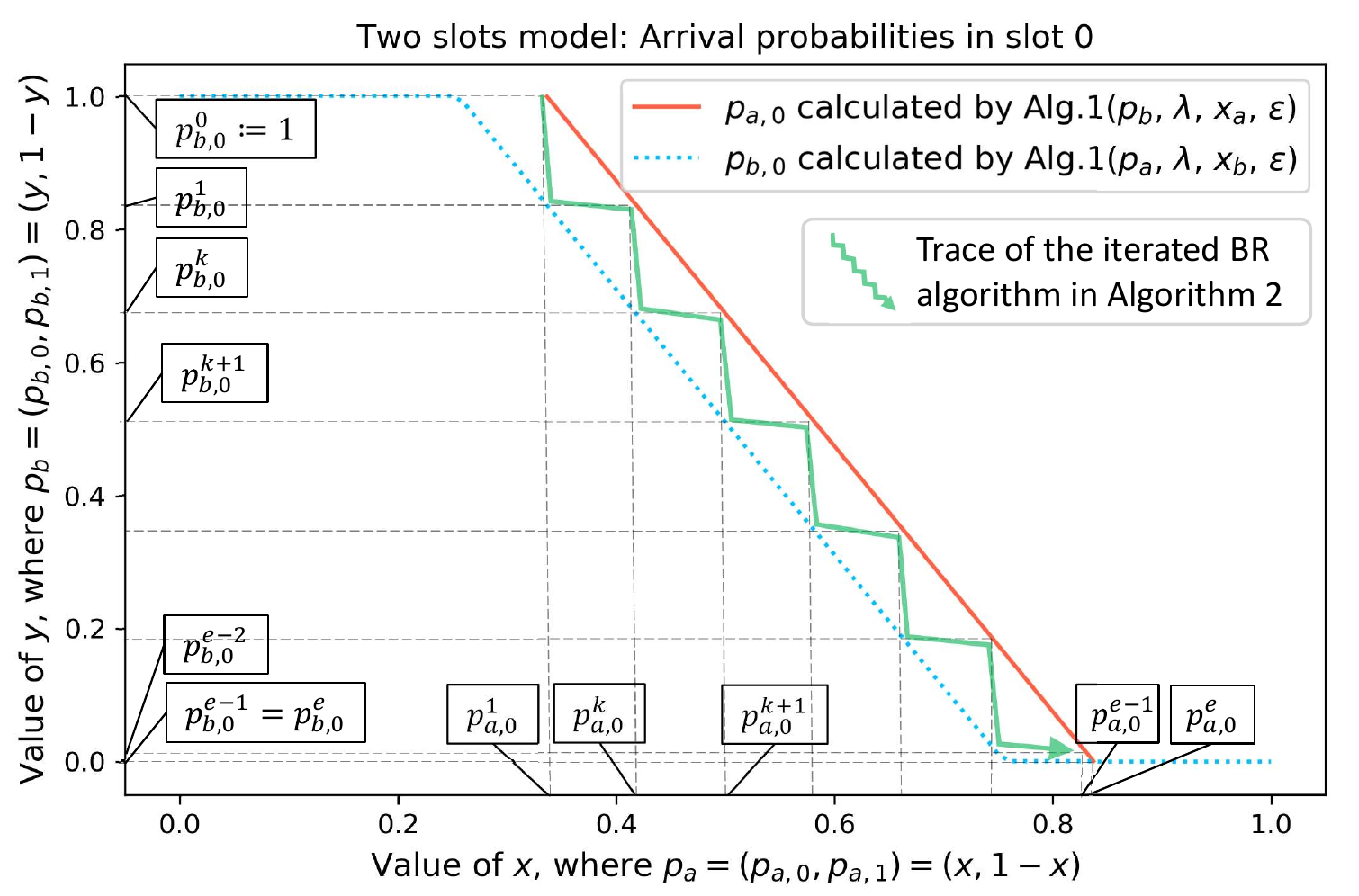}
	\caption{An example for the best response of type $i$ ($i = a,b$) customers against $\vc{p}_{-i}$, where $(T+1, \tau, \lambda_{a}, \lambda_{b}, \chi_{a}, \chi_{b}) = (2, 30, 10, 5, 3, 2)$  and each service time follows a geometric distribution.}
	\label{fig:2slotsbrexample}
\end{figure}
\end{remark}

\subsection{Existence of equilibrium arrival distributions}\label{sec:discrete_NE_existence}

In \cite{SMF2019} it was shown that an equilibrium exists in the single-type game. Furthermore, it was conjectured that the equilibrium is unique, and this conjecture was strengthened by numerical analysis. The existence result can be extended to the two-type game by applying Kakutani's fixed point theorem. The issue of uniqueness is an open problem even for the single-type game, but there is reason to believe the conjecture holds in that case. In the multi-class setting it is less clear if the solution should be unique. Recall that for the deterministic fluid model Theorem \ref{thm:NE_fluid} showed that multiple equilibria are possible for some parameter settings.

\begin{proposition}\label{prop:existence}
A symmetric (within type) pair of equilibrium arrival distributions $(\vc{p}_{a}^{e},\vc{p}_{b}^{e})$ exists for any game parameters. 
\end{proposition}
\begin{proof}
For type $i\in\mathcal{C}$ customers, given any arrival distribution $\vc{p}_{-i}$ of customers with the other type, denote the outcome of Algorithm~1 by $\mathcal{BR}_i(\vc{p}_{-i}):[0,1]^{T+1}\to[0,1]^{T+1}$. The algorithm always returns a valid distribution due to the existence result of \cite{SMF2019} for the single-type game. Let 
\[
\mathcal{BR}(\vc{p}_a,\vc{p}_b):=(\mathcal{BR}_a(\vc{p}_b),\mathcal{BR}_b(\vc{p}_a)):[0,1]^{T+1}\times[0,1]^{T+1}\to[0,1]^{T+1}\times[0,1]^{T+1} \ ,
\]
denote the two dimensional mapping of the pair of distributions $(\vc{p}_a,\vc{p}_b)$ to their respective $\mathcal{BR}_i$ sets. By applying Kakutani's fixed point theorem (e.g. Lemma 20.1 of \cite{book_OR1994}) we can conclude that $\mathcal{BR}$ has a fixed point $(\vc{p}_a,\vc{p}_b)=\mathcal{BR}(\vc{p}_a,\vc{p}_b)$. In particular:
\begin{enumerate}
\item[(i)] $[0,1]^{T+1}$ is a compact and convex set,
\item[(ii)] $\mathcal{BR}(\vc{p}_a,\vc{p}_b)$ is a single point in $[0,1]^{T+1}\times[0,1]^{T+1}$, hence it is trivially a convex set for any $(\vc{p}_a,\vc{p}_b)$,
\item[(iii)] and $\mathcal{BR}$ has a closed graph (because solutions of \eqref{eqn:NE_Arrival_Distributions} are continuous with respect to all coordinates of $\vc{p}_i$).
\end{enumerate} 
\end{proof}

\subsection{Numerical analysis}\label{sec:numerical}

In this section, we show several examples of the equilibrium arrival distributions for different service-time distributions for the discrete-time system.
The equilibrium arrival distributions are computed using Algorithm 2. 
We demonstrate that Algorithm 2 gives equilibrium arrival distributions with a similar form as the equilibrium arrival distributions of the fluid model for the non-degenerate cases in Theorem~\ref{thm:NE_fluid} of Section 4. To this end, the intervals for $T$ in Theorem~\ref{thm:NE_fluid}~(i)--(iv) are denoted by  $(0, \xi_{1}]$, $(\xi_{1}, \xi_{2})$, $[\xi_{2}, \xi_{3}]$ and $(\xi_{3}, \xi_{4}]$, respectively, where
\[
(\xi_1, \xi_2, \xi_3, \xi_4) :=
\left(
\frac{\lambda_a + \lambda_b}{2 \mu_b}, 
\frac{\lambda_a + 2 \lambda_b}{2 \mu_b}, 
\frac{\lambda_a}{2 \mu_a} + \frac{\lambda_b}{\mu_b},
\frac{\lambda_a}{\mu_a} + \frac{\lambda_b}{\mu_b}  
\right).
\]

In the following examples (see Figures~\ref{fig:deterministic}--\ref{fig:mix_of_geometric}), both the time unit and slot length are set to $1$[min]; the mean number of arriving customers with beliefs $a$ and $b$ are given by $\vc{\lambda} := (\lambda_a,\lambda_b) = (50,50)$; the pair of the mean-service times are given by $(\chi_a,\chi_b) = (4,2)$.
We consider the following three cases for the service-time distributions: $\vc{x}_{i}$ ($i \in \calC$) is a deterministic distribution (see Figure~\ref{fig:deterministic}); $\vc{x}_{i}$ ($i \in \calC$) is  a geometric distribution (see Figure~\ref{fig:geometric}); $\vc{x}_{i}$ ($i \in \calC$) is a mixture of two geometric distributions (see Figure~\ref{fig:mix_of_geometric}).
The parameters of the mixture of two geometric distributions are chosen so that their coefficient of variations (CV’s, for short) are twice of that of the geometric distributions. 
Under the above parameter settings, we have $(\xi_1,\xi_2,\xi_3,\xi_4 ) = (100,150,200,300)$.
In Figures~\ref{fig:deterministic}--\ref{fig:mix_of_geometric}, the following four cases for the length of the acceptance period are considered:
\begin{align*}
\tau \times (T+1) = 
\begin{cases}
	60 \in (0, \xi_{1}]  & (\text{i.e., Case (i) in Theorem~\ref{thm:NE_fluid}}) \\
	120 \in (\xi_{1}, \xi_{2}] & (\text{i.e., Case (ii) in Theorem~\ref{thm:NE_fluid}})\\
	180 \in (\xi_{2}, \xi_{3}] & (\text{i.e., Case (iii) in Theorem~\ref{thm:NE_fluid}})\\	
	240 \in (\xi_{3}, \xi_{4}] & (\text{i.e., Case (iv) in Theorem~\ref{thm:NE_fluid}})
\end{cases}
\end{align*}
In each case, the equilibrium arrival distributions are computed by
\begin{align}
\label{eqn:Alg_2}
(\vc{p}_{a}^{e}, \vc{p}_{b}^{e})
 = \mathbf{Alg.2}(\vc{\lambda}, \vc{x}_{a}, \vc{x}_{b}, \epsilon, \delta),
\end{align}
where $\epsilon = \delta := 10^{-5}$.

\begin{figure}[H]
	\centering
	\includegraphics[width=\linewidth]{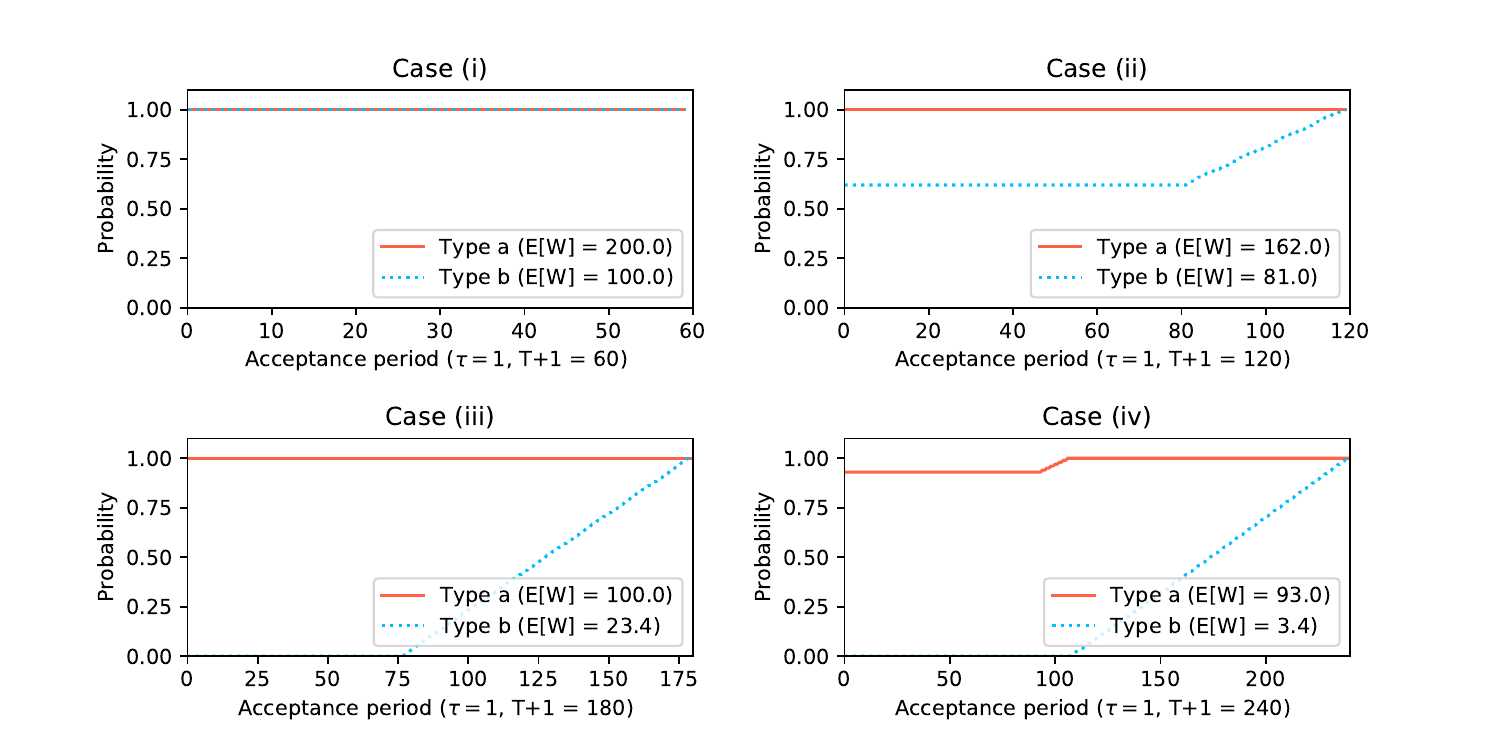}
	\caption{cdfs of the equilibrium arrival distributions computed by (\ref{eqn:Alg_2}), where the service time follows a deterministic distribution}
	\label{fig:deterministic}
\end{figure}

\begin{figure}[H]
	\centering
	\includegraphics[width=\linewidth]{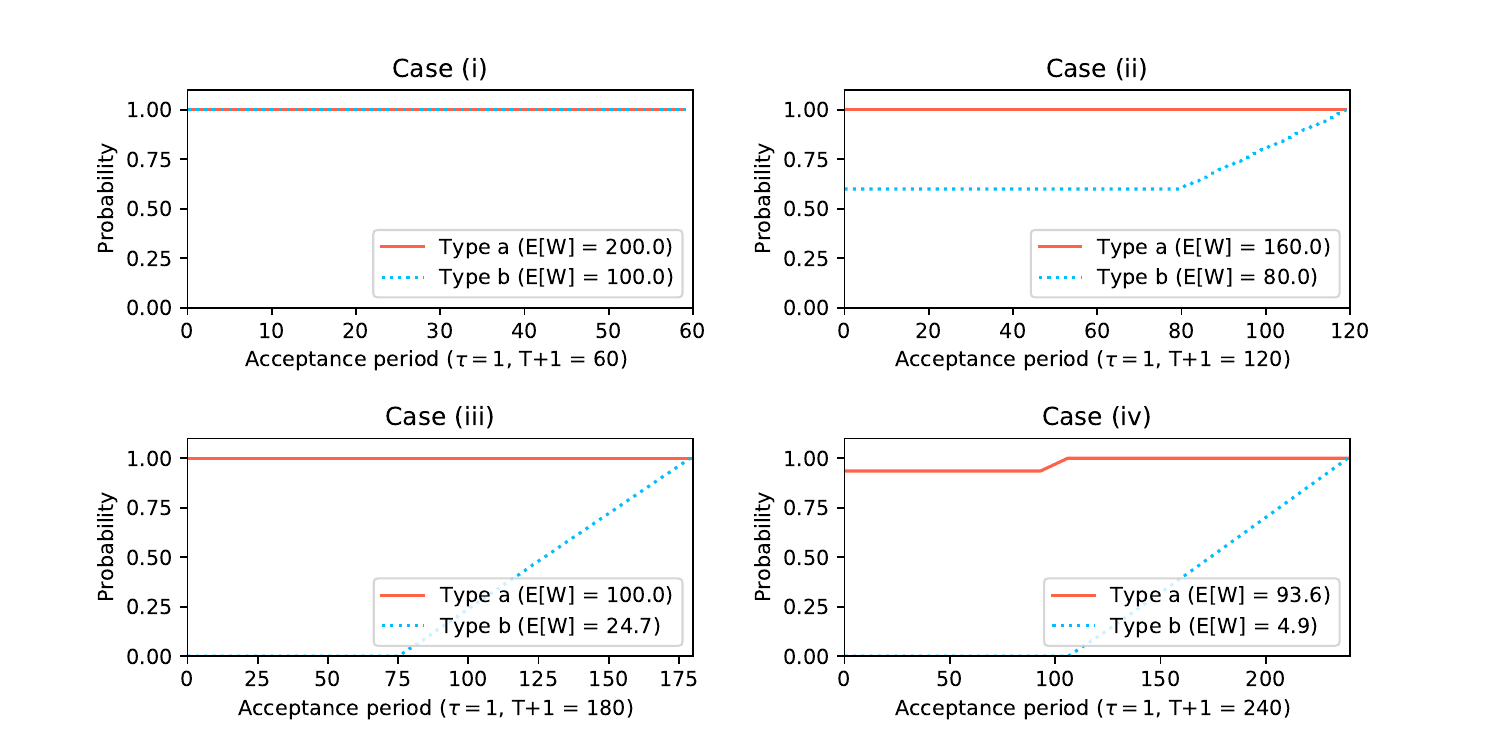}
	\caption{cdfs of the equilibrium arrival distributions computed by (\ref{eqn:Alg_2}), where the service time follows a geometric distribution, where $\text{CV}_a = 0.87$, $\text{CV}_b = 0.71$}
	\label{fig:geometric}
\end{figure}

\begin{figure}[H]
	\centering
	\includegraphics[width=\linewidth]{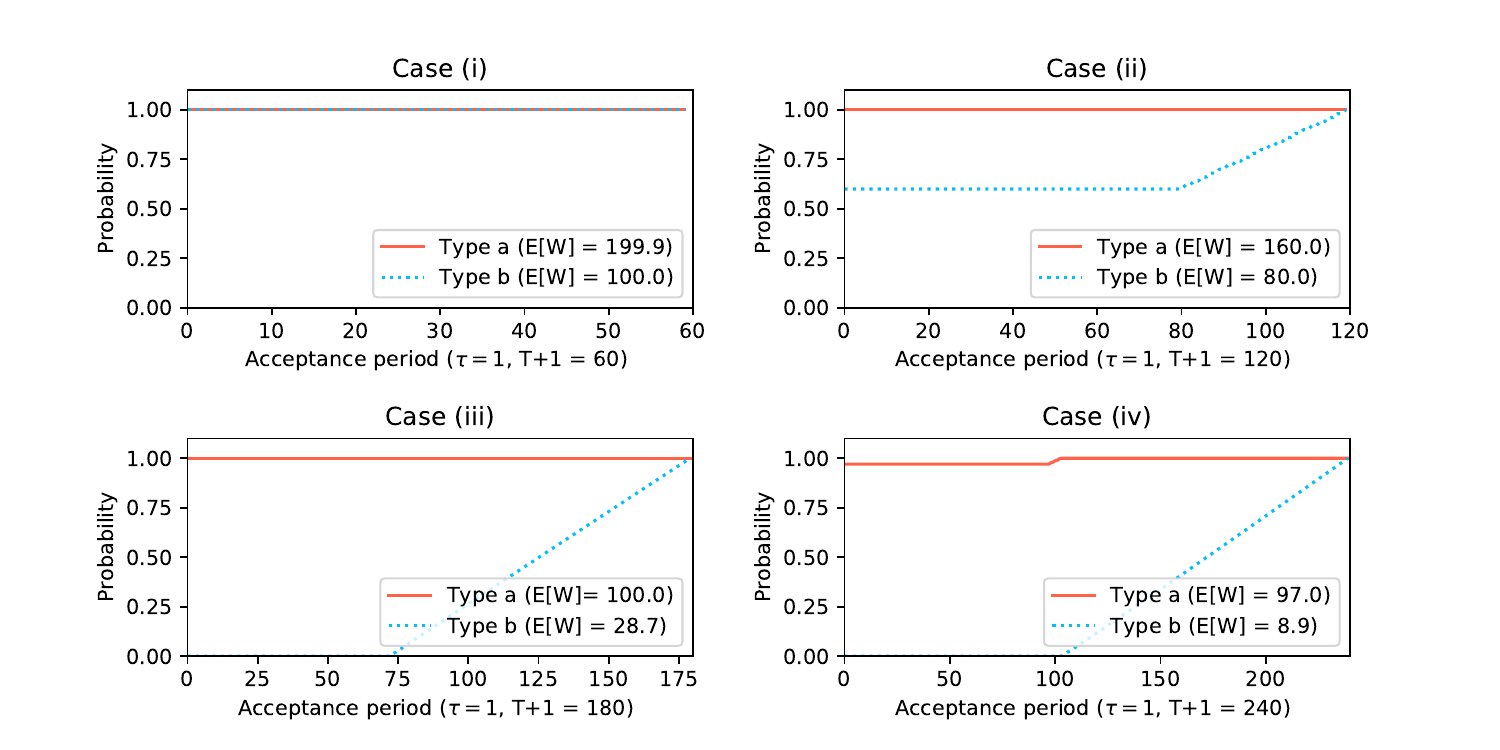}
	\caption{cdfs of the equilibrium arrival distributions computed by (\ref{eqn:Alg_2}), where the service time follows a mixtures of two geometric distributions, where $\text{CV}_a = 1.74$, $\text{CV}_b = 1.42$}
	\label{fig:mix_of_geometric}
\end{figure}

In all numerical examples Algorithm 2 converged to a unique equilibrium solution.
The cdf's of the equilibrium arrival distributions for the four cases are illustrated in Figures~\ref{fig:deterministic}--\ref{fig:mix_of_geometric}. It can be seen that the solutions have a similar form to the corresponding explicit linear solutions in cases (i)—(iv) in Theorem \ref{thm:NE_fluid} (which are illustrated in Figure \ref{fig:fluid}).  Furthermore, we can see that the expected waiting time increases as the CV of the service time gets large. 

\section{Agent-based model}\label{sec:learn}

We now present a discrete-time agent-based model (ABM for short) as a behavioral alternative to the equilibrium analysis of Section \ref{sec:discrete}. Our ABM extends that of \cite{SMF2019} to account for the scenario of server uncertainty and noisy signals as modeled in Section \ref{sec:signals}. The customers arrive to the system repeatedly and choose their arrival time slot according to a decision model that combines random exploration and minimization of the estimated waiting time according to the past delays experienced by the individual customer. We then compare its long-term averaged arrival distributions with the equilibrium calculated by Algorithm 2. Recall the information and rationality assumptions discussed in Section \ref{sec:belief}. The rational models of the previous sections corresponded to Assumptions (FR)--\textit{fully informed rational customers}, and (BR)--\textit{customers with bounded rationality or limited information. This section deals with an example of assumption (AM)--\textit{Agent based model}, i.e., with customers that have no knowledge of the system parameters (or cannot make the necessary waiting time computations, but rather react to their own past delay experiences.}  

We assume that there is a finite large pool of $N$ potential customers labeled as the numbers in $\calN := \{ 1,2,\ldots,N \}$ that join independently at any given day with a probability $\delta_N$ such that $\delta_N\to 0$ as $N\to\infty$ and $N\delta_N=\lambda$. The total arrivals in a given day is therefore $A\sim\mathrm{Bin}(N,\delta_N)$ and can be approximated by Poisson$(\lambda)$ for large $N$. This formulation enables giving customers identities which is important for the learning framework as each customer has their own personal history of waiting time observations. As in Section~\ref{sec:signals}, the system dynamics ``on each day'' is described as follows.  On day $d$ ($d \ge 1$), the service mode is denoted by $M_{d} \in \calC$, where
\[
 \P(M_{d} = a) = p, \quad \P(M_{d} = b) = 1-p\ .
\]
When the service mode $M_{d}$ is given by $i \in \calC$, during day $d$, the system offers iid service times following the distribution $\vc{x}_{i}$ with mean $\chi_{i}$. 
The signal (i.e., belief) on the service mode received by customer $k$ ($k \in \calN$), denoted by $Y_{d}^{(k)}$, is given by
\[
 Y_{d}^{(k)} = M_{d} S_{d}^{(k)} + (a + b - M_{d}) (1 - S_{d}^{(k)})\ ,
\]
where $S_{d}^{(k)} \in \{0,1\}$ be an iid Bernoulli random variable with $\P(S_{d}^{(k)} = 1) = q$. 
The server state and signals on any day are independent of each other and of the states and signals in all other days. We assume customers do not know $p$ or $q$ but they do know that $q>\frac{1}{2}$, and so the signal is informative. The acceptance period is the discrete grid $\mathcal{T}=\{0,\tau,\ldots,T\tau\}$ as in Section \ref{sec:discrete}.

On day $d\geq 1$ an arriving customer $k$ observes information $Y_{d}^{(k)} = i \in \calC$, and chooses a time slot $t_d^{(k,i)} = t \in \mathcal{T}$ according to a decision rule that will be elaborated on later, and records the waiting time $W_{d,t}^{(k,i)}$. Let $\bar{\vc{W}}_{d}^{(k)} := (\bar{W}_{d,t}^{(k,i)} ; i \in \calC, t \in \calT)$ denote the $2 \times (T+1)$ matrix of customer $k$'s average waiting times, where $\bar{W}_{d,t}^{(k,i)}$ is the average waiting time that customer $k$ with belief $i$ experiences when she/he arrives at slot $t$ during first $d$ days. Define $\bar{\vc{W}}_{0}^{(k)}$ as a zero matrix, i.e., $\bar{W}_{0,t}^{(k,i)} := 0$ for $i \in \calC$ and $t \in \calT$. The daily decisions involve randomization between past experience and random experimentation. We denote the experimentation probability as a function of trials as
\begin{equation}\label{eq:sigmoid}
\theta(x)=e^{\frac{c_1}{1-e^{c_2 x}}}\ ,
\end{equation}
where $c_1,c_2>0$ (see \cite{SMF2019} for further discussion on these learning dynamics). On day $d$, an arriving customer $k$ with belief $i$ ($i \in \calC$) chooses her/his arrival slot uniformly from $\calT$ with probability $1 - \theta(A^{(k,i)}(d))$, where $A^{(k,i)}(d)$ denotes the number of customer $k$' arrivals with belief $i$ during the first $d-1$ days. 
On the other hand, with probability $\theta(A^{(k,i)}(d))$, she/he chooses the time slot with the minimal average experienced waiting time during the first $d-1$ days.

Therefore the decision rule can be written as follows,
\begin{equation}\label{eq:decision}
t_d^{(k,i)}=\left\lbrace\begin{array}{ll}
\mathrm{Uniform}(\mathcal{T}), & \mathrm{w.p.} \quad 1 - \theta(A^{(k,i)}(d))\ , \\
\argmin_{t \in \calT} W_{d-1,t}^{(k,i)}, & \mathrm{w.p.} \quad \theta(A^{(k,i)}(d))\ .
\end{array}\right.
\end{equation}

By \eqref{eq:sigmoid} and \eqref{eq:decision} we have that as $d$ grows the probability of uniform exploration diminishes and the minimal average waiting time is chosen with an arbitrarily high probability. Consider a tagged customer $k$ and her/his experienced average waiting time $W_{d,t}^{(k,i)}$ on days with belief $i\in\mathcal{C}$ such that the specific slot $t$ has been sampled. Clearly, as the number of days grows all customers will join and receive both types of signals infinitely often. The limiting proportion of the number of times slot $t\in\mathcal{T}_\infty$ is chosen with belief $i$ is given by
\begin{equation}
\label{eqn:Arrival_Dist_DLM}
p_t^{(k,i)} := \lim_{D\to\infty} \frac{1}{\max\{ 1, A^{(k,i)}(D)\}} \sum_{d=1}^D \mathbf{1}(\text{customer $k$ arrives on day $d$}, t_d^{(k,i)} = t)\ ,
\end{equation}
and her/his limiting average waiting time during the acceptance period, denoted by $\ol{w}^{(k,i)}$, is given by
\begin{align}
\label{eqn:}
 \ol{w}^{(k,i)} = \sum_{t \in \calT} p_t^{(k,i)} \bar{W}_{t}^{(k,i)},
\end{align}
where $\bar{W}_{t}^{(k,i)} := \lim_{d \to \infty} \bar{W}_{d,t}^{(k,i)}$ for $t \in \calT$. In the subsequent two sections, we consider the averaged arrival distributions and mean waiting times over all customers as follows.
\begin{align}
\label{eqn:ABM_averaged_arrival_dist}
 \vc{\ol{p}}_{N}^{(i)} := \frac{1}{N} \sum_{k=1}^{N} (p_t^{(k,i)} ; t \in \calT)\ ,
 \qquad
 \ol{w}_{N}^{(i)} := \frac{1}{N} \sum_{k=1}^{N} \ol{w}^{(k,i)}
\end{align}
for each belief $i \in \calC$.

\subsection{Numerical comparison between ABM and bounded rationality}\label{sec:NE_no_bias_vs_ABM}

In what follows, we compare the cdf of arrival time and mean waiting time calculated by the ABM with the one calculated by Algorithm 2. In the latter computation of the Nash equilibrium we make assumption (BR) in Section~\ref{sec:belief}: customers fully trust their signal and assume all customers have the corresponding service time distribution (without the posterior update). The results which are computed by Algorithm 2 and the ABM are referred to as (BR) and (AM), respectively.

The parameter settings of the agent-based model are given as follows: the length of the acceptance period is set to $60$ minutes and is divided into $(T+1)=20$ slots, i.e., the length of each slot is given by $\tau = 3$ [min]; the mean number of population is given by $\lambda = 10$; the server uncertainty and the strength of the signal are given by $p = 0.5$ and $q = 0.9$, respectively; the mean service-times in modes $a$ and $b$ are given by $\chi_a = 4$ and $\chi_b = 2$, respectively. 
With these parameter settings, the mean number of arrival customers with beliefs $a$ and $b$ are given by $\lambda_a = \lambda (pq + (1-p)(1-q)) = 5$ and $\lambda_b = \lambda (p(1-q) + (1-p)q) = 5$, respectively.
As in Section~\ref{sec:numerical}, we consider the same three service-time distributions, i.e.,  deterministic, geometric distribution, and a mixture of two geometric distributions.
For these three distributions, Figures~\ref{fig:deterministicnewithoutbiasvsdl}--\ref{fig:mixofgeometricnewithoutbiasvsdl} respectively show the comparison between the cdf calculated by Algorithm 2 and the ones obtained by simulations of the ABM, where every customer arrives at the system $1,000$ times on average.

\begin{figure}[H]
	\centering
	\includegraphics[width = 0.9\linewidth]{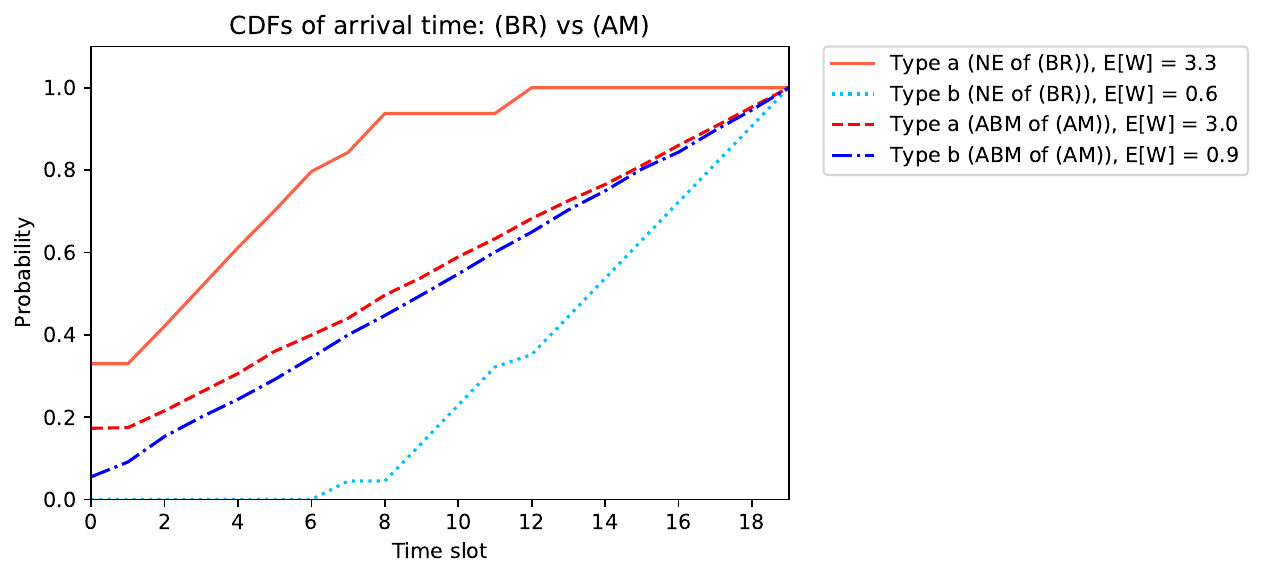}
	\caption{Comparison between cdfs of the equilibrium arrival distributions calculated by (\ref{eqn:Alg_2}) and ones obtained by the averaged arrival distributions (\ref{eqn:ABM_averaged_arrival_dist}) calculated by ABM, where the service time follows a deterministic distribution}
	\label{fig:deterministicnewithoutbiasvsdl}
\end{figure}

\begin{figure}[H]
	\centering
	\includegraphics[width = 0.9\linewidth]{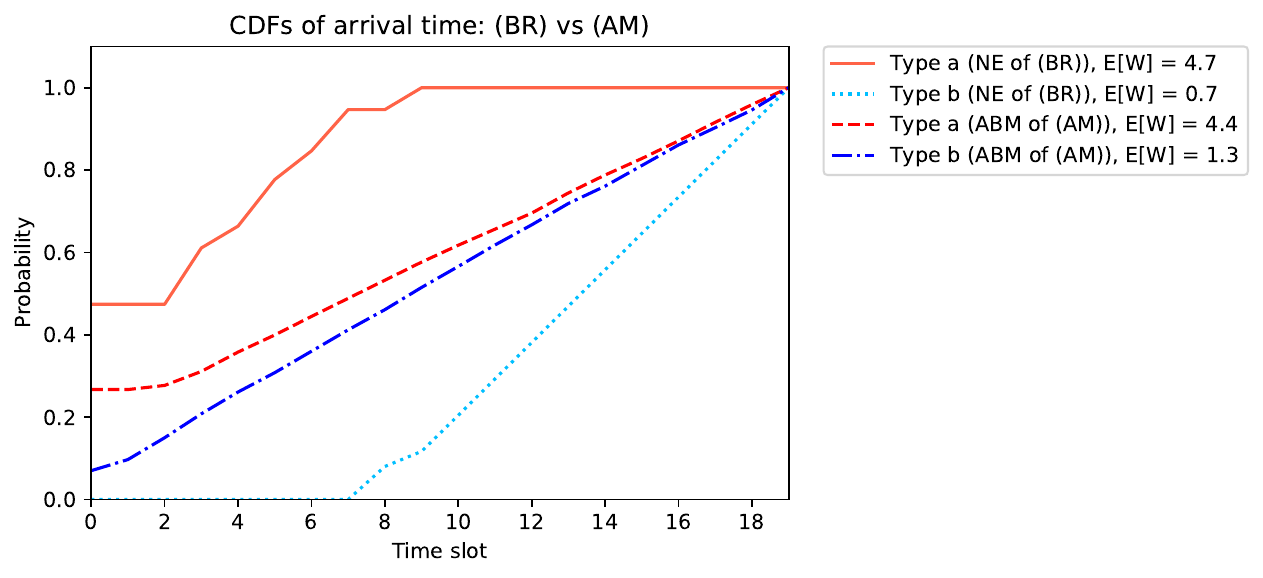}
	\caption{Comparison between cdfs of the equilibrium arrival distributions calculated by (\ref{eqn:Alg_2}) and ones obtained by the averaged arrival distributions (\ref{eqn:ABM_averaged_arrival_dist}) calculated by ABM, where the service time follows a geometric distribution}
	\label{fig:geometricnewithoutbiasvsdl}
\end{figure}

\begin{figure}[H]
	\centering
	\includegraphics[width = 0.9\linewidth]{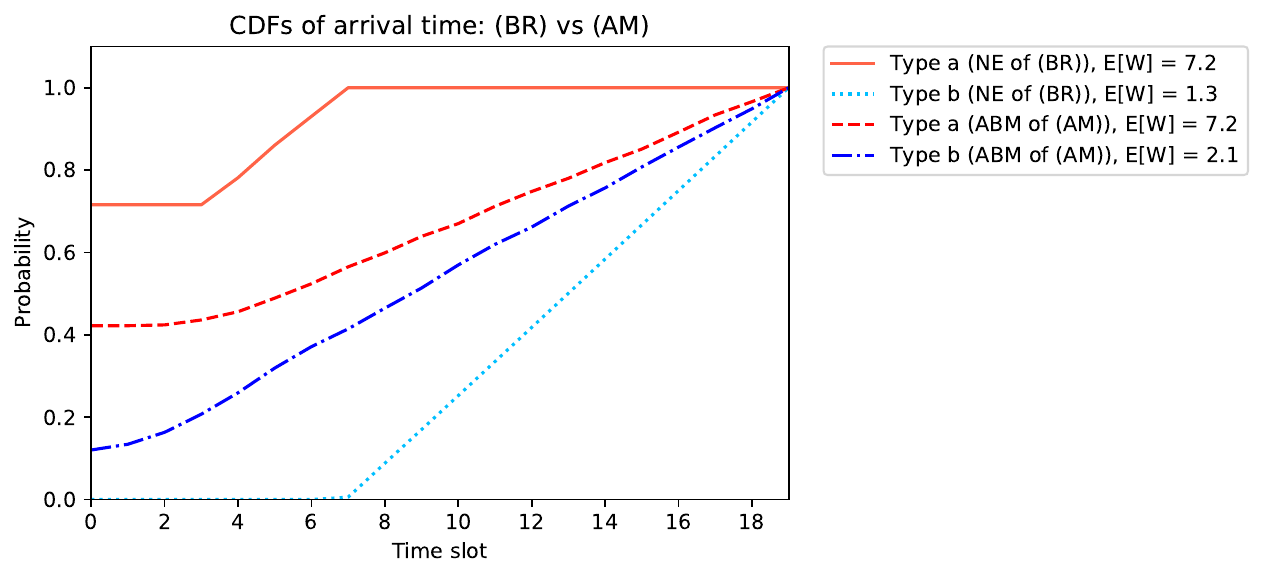}
	\caption{Comparison between cdfs of the equilibrium arrival distributions calculated by (\ref{eqn:Alg_2}) and ones obtained by the averaged arrival distributions (\ref{eqn:ABM_averaged_arrival_dist}) calculated by ABM, where the service time follows a mixtures of two geometric distributions}
	\label{fig:mixofgeometricnewithoutbiasvsdl}
\end{figure}

From Figures~\ref{fig:deterministicnewithoutbiasvsdl}--\ref{fig:mixofgeometricnewithoutbiasvsdl}, we can see the outcome is quite different, i.e., the gap between the cdfs of (BR) and (AM) is quite large for each belief ($a$ and $b$). Specifically, the averaged arrival distributions of the ABM (see (AM) in the figures) tend to spread over the acceptance period. On the other hand, for the equilibrium arrival distributions calculated by Algorithm 2 (see (BR) in the figures), the arrivals of beliefs $a$ and $b$ customers tend to the equilibrium arrival distributions of beliefs $a$ and $b$ tend to lie in the first and the last half of the acceptance period, respectively.

The above differences between the arrival distributions of the ABM and equilibrium models also have implications for social welfare. In the homogenous case, when customers are more spread out then waiting times are decreased and thus social welfare is higher (see \cite{HK2011} for detailed analysis of socially optimal arrival strategies). In the case of heterogeneous beliefs customers tend to arrive closer to their own type but further away from the other type (or on completely disjoint intervals as in the fluid model), thus the effect on social welfare may be negative or positive, depending on the exact system parameters. In Figures~\ref{fig:deterministicnewithoutbiasvsdl}--\ref{fig:mixofgeometricnewithoutbiasvsdl} we observe that in all cases the mean waiting time is higher (or equal) for type $a$ customers in the bounded rationality equilibrium than in the long term average of the ABM, while the opposite is true for type $b$ customers, i.e., a lower mean waiting time in equilibrium. For example, in Figure~\ref{fig:geometricnewithoutbiasvsdl}, the belief $b$'s mean waiting time in ABM is about twice as high than that of NE, while for type $a$ customers the mean waiting time in ABM is slightly lower than that of the NE (less than 10\% lower in the ABM). Our numerical analysis therefore suggests that in terms of social welfare, customers who just react to their past experiences can sometimes be better off than customers with (limited) knowledge of the system parameters and dynamics. A possible explanation for this gap is that in the ABM customers are eventually estimating the correct properties of the system, while the customers with bounded rationality are computing the `wrong' equilibrium. The sequel shows that the gap is indeed diminished when customers are fully rational.

\subsection{Numerical comparison between ABM and full rationality}\label{sec:NE_with_bias_vs_ABM}

The queueing game under the bounded rationality assumption (BR) does not take into account the information structure given by the probabilities $p$ and $q$. A more appropriate comparison with the ABM (AM) is therefore with the queueing game under assumption (FR) that assumes customers make the posterior update of expectations as detailed in Section~\ref{sec:signals}. For an arrival customer with belief $i$ ($i \in \calC$) in the ABM, the posterior arrival rates of the other customers with beliefs $a$ and $b$ are given by (\ref{eqn:Bias_tagged_class_a}) and (\ref{eqn:Bias_tagged_class_b}),
$\vc{\nu}_{i} := \lambda (\alpha_{ai}, \alpha_{bi})$.

First consider the service time distribution of an arrival customer with belief $i$ ($i \in \calC$). Similar to (\ref{eqn:Biased_service_time_class_i}), let $\vc{z}_{i}$ ($i \in \calC$) denote the service time with mean $\zeta_{i}$, then
\begin{align*}
 \vc{z} := (z_a, z_b) = (\eta_{aa} \vc{x}_{a} + \eta_{ba} \vc{x}_{b}, \eta_{ab} \vc{x}_{a} + \eta_{bb} \vc{x}_{b})\ .
\end{align*}
Therefore, in the view point of belief $a$ (belief $b$) arrival customers, the arrival rates to the system are $\vc{\nu}_{a}$ ($\vc{\nu}_{b}$), and the service time random variables of beliefs $a$ and $b$ customers are distributed as $\vc{z}$. Furthermore, the equilibrium arrival distributions, denoted by $(\hat{\vc{p}}_{a}^{e}, \hat{\vc{p}}_{b}^{e})$, are computed by
\begin{align}
\label{eqn:Alg_2_bias}
(\hat{\vc{p}}_{a}^{e}, \bullet) = \mathbf{Alg.2}(\vc{\nu}_{a}, \vc{z}, \epsilon, \delta),\qquad
(\bullet, \hat{\vc{p}}_{b}^{e}) = \mathbf{Alg.2}(\vc{\nu}_{b}, \vc{z}, \epsilon, \delta)
\end{align}

In what follows, we compare the cdf of arrival time and mean waiting time of the ABM with those obtained for the equilibrium by (\ref{eqn:Alg_2_bias}). In line with the customer rationality and knowledge assumptions in Section~\ref{sec:belief}, the results which are computed by Algorithm 2 are referred to as (FR).
The parameter settings are the same as that of the preceding section, i.e., $\lambda = 10$, $(p,q) = (0.5, 0.9)$ and $(\chi_{a}, \chi_{b}) = (4.0, 2.0)$, then we have $\vc{\nu}_{a} = (8.2, 1.8)$, $\vc{\nu}_{b} = (1.8, 8.2)$ and $(\zeta_{a}, \zeta_{b}) = (3.8, 2.2)$.

The results are displayed in Figures~\ref{fig:deterministicnewithbiasvsdl}--\ref{fig:mixofgeometricnewithbiasvsdl}, and we can see that the resulting cdfs of the arrival distributions and mean waiting times of the equilibrium (FR) are closer to the outcome of the ABM (AM). However, the distribution resulting from the learning dynamics of the ABM is still more spread out than the equilibrium distribution, although the difference is much smaller than that observed in the comparison with the bounded rationality equilibrium of the previous section. In particular, the optimistic type $b$ customers consistently arrive at $t = 0$ with
positive probabilities even when the equilibrium prediction is that they only start arriving later. On the other hand, pessimistic type $a$ customers arrive at the last slots with positive probability even though they would not do so in equilibrium. These discrepancies are due to the estimation bias resulting from ignoring the system parameters and information structure.

\begin{figure}[H]
	\centering
	\includegraphics[width = 0.9\linewidth]{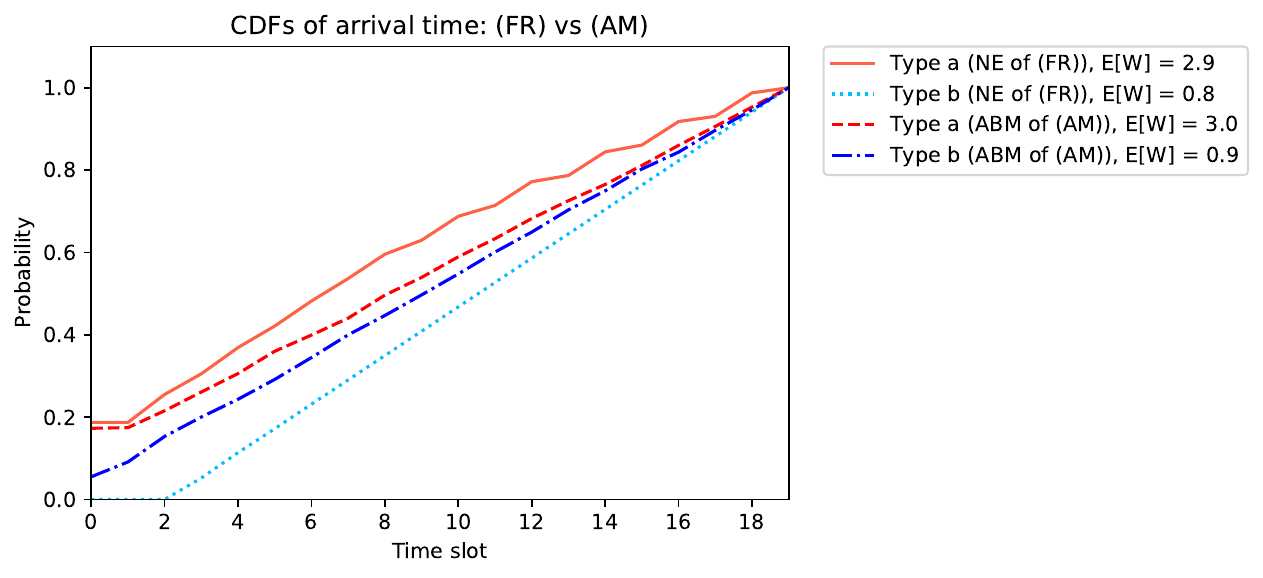}
	\caption{Comparison between cdfs of the equilibrium arrival distributions calculated by (\ref{eqn:Alg_2_bias}) and the averaged arrival distributions (\ref{eqn:ABM_averaged_arrival_dist}) calculated by ABM, where the service time follows a deterministic distribution.}
	\label{fig:deterministicnewithbiasvsdl}
\end{figure}

\begin{figure}[H]
	\centering
	\includegraphics[width = 0.9\linewidth]{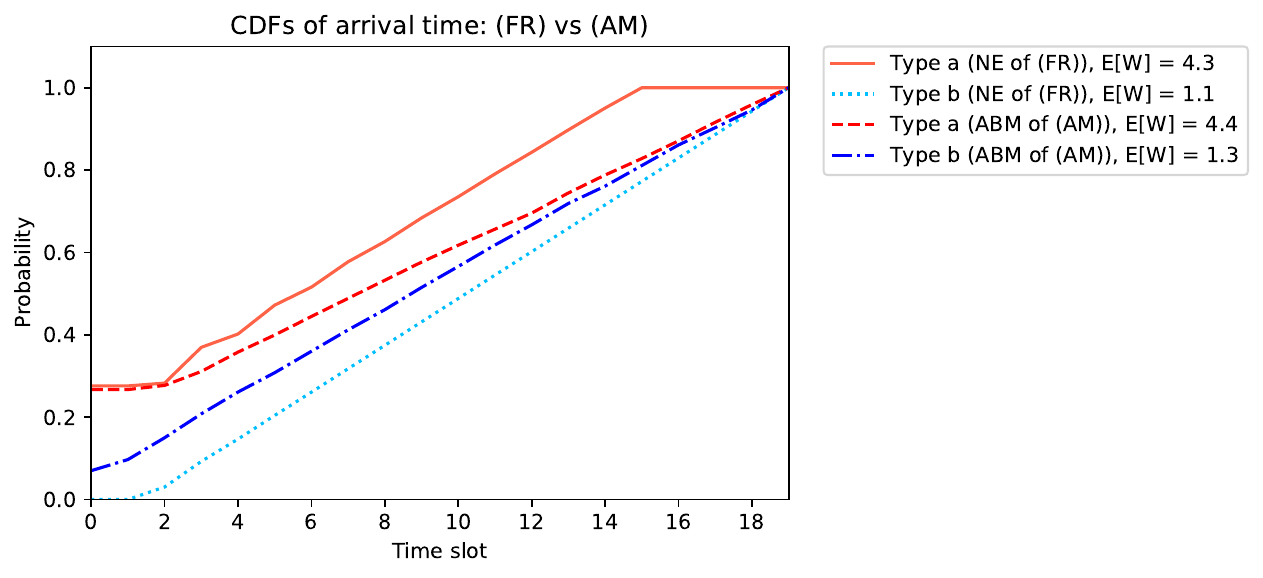}
	\caption{Comparison between cdfs of the equilibrium arrival distributions calculated by (\ref{eqn:Alg_2_bias}) and the averaged arrival distributions (\ref{eqn:ABM_averaged_arrival_dist}) calculated by ABM, where the service time follows a geometric distribution.}
	\label{fig:geometricnewithbiasvsdl}
\end{figure}

\begin{figure}[H]
	\centering
	\includegraphics[width = 0.9\linewidth]{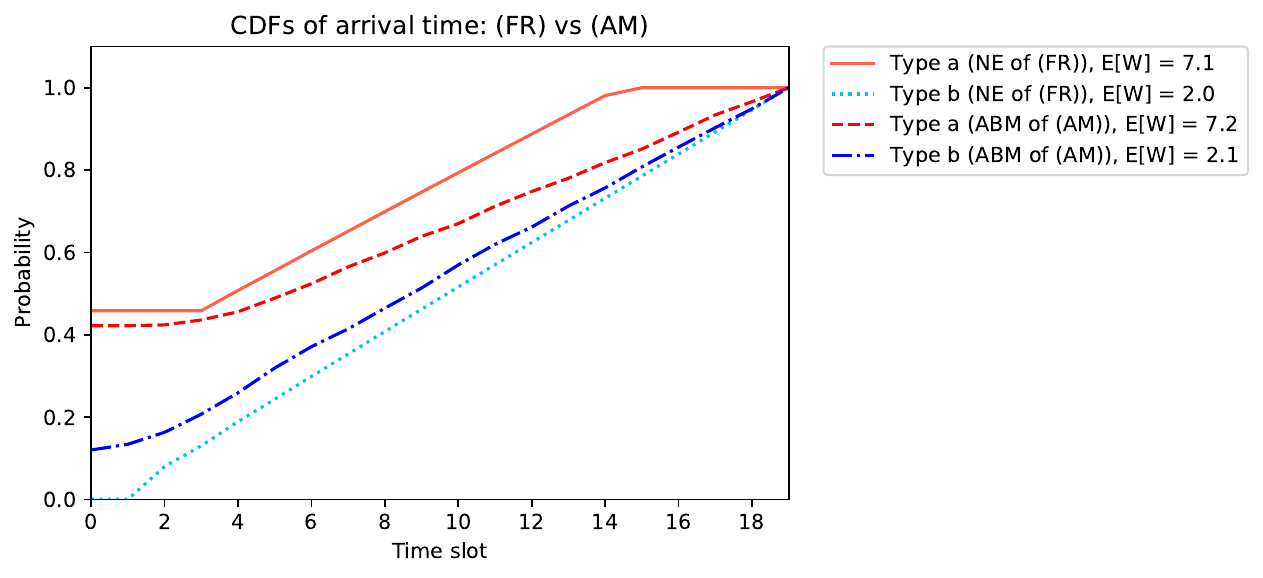}
	\caption{Comparison between cdfs of the equilibrium arrival distributions calculated by (\ref{eqn:Alg_2_bias}) and the averaged arrival distributions (\ref{eqn:ABM_averaged_arrival_dist}) calculated by ABM, where the service time follows a mixtures of two geometric distributions.}
	\label{fig:mixofgeometricnewithbiasvsdl}
\end{figure}

\section{Conclusion}\label{sec:conclusion}

This paper has introduced a general framework to analyze the arrival distributions to a bottleneck queue with multiple customer types that differ in their belief regarding the service time distribution. Two models were presented. First, a game where rational customers that know all of the system parameters, including the mechanism driving uncertainty, with the solution given by a Nash equilibrium. Second, a dynamic learning model in which customers adapt their decisions based on past experience. Constructive procedures have been provided for the computation of the arrival distribution in both cases.

Our framework can be extended in several natural directions. From an economic point of view, more elaborate cost functions can be considered, e.g. tardiness or order penalties and rewards for service. The customers may have heterogeneous features that may be due to individual preferences (e.g. waiting costs) or partial information (e.g. random value of service). Such games may also combine multiple $n>2$ types of customers or beliefs. From a queueing perspective, different system configurations may be considered, e.g. a multi-server system with uncertainty regarding the number of servers. Both our equilibrium and learning analysis can be extended to the above settings.

The learning model of Section \ref{sec:learn} can be seen as a simple first attempt at formulating interesting learning dynamics for the arrival time problem to a bottleneck queue. There are several issues that are perhaps of interest to explore further. Notably, accurate asymptotic analysis of the average arrival distribution. This can be combined with further investigation of the bias brought on by the average waiting time estimation, and can it be improved by decision rules that are still reasonably simple. For example, we assumed that the information quality is unknown, i.e. customers just know that $q>\frac{1}{2}$. If we further assume that customer observe service times, then $q$ can be estimated as follows. Denote by $X_{i}^{d}$ the service time observed on day $d$ when the belief was $y\in\mathcal{C}$, and let $\bar{X}_i$ be the respective average service time on days (up to day $d$) that the service belief was $y$. By the law of large numbers, $\bar{X}_{a}^{(d)}\parrow q\chi_a+(1-q)\chi_b$, and this yields an asymptotically unbiased estimator
\[
\bar{q}_a^{(d)}=\frac{\bar{X}_a^{(d)}-\chi_b}{\chi_a-\chi_b}\ .
\]
An interesting question is then if the above can be used to make a simpler decision role that relies on multiple estimators for the system parameters, i.e. arrival and service rates as well as the information strength. Of course, such a learning process again assumes customers have the ability and resources to make more elaborate computations for the decision making, but it is interesting to examine what can be achieved with slightly more sophisticated estimation and decision rules.

\section*{Acknowledgments}\label{sec:acknowledgments}
The authors wish to thank Refael Hassin for his helpful comments on a preliminary draft of this manuscript. The first author's work was supported by the NWO Gravitation project {\sc Networks}, grant number 024.002.003. The second author's work was supported by JSPS KAKENHI Grant No. JP18K11186.

\begin{appendices}

\section{Detailed proof of Lemma \ref{lemma:queue_order}}\label{sec:appn_A}

To prove Lemma \ref{lemma:queue_order} we construct a coupling of the virtual waiting time and queueing processes for both types of customers and show that for every sample path type $a$ customers face a longer queue and waiting time than type $b$ customers for any arrival time $t\in\mathcal{T}$.

\begin{lemma}\label{lemma:workload_dominance}
Let $X_{i,k}\sim\exp(\mu_i)$ denote the job size of the $k$'th arrival when the service distribution is of type $i\in\mathcal{C}$ customers. If $\mu_a<\mu_b$ then the virtual waiting time $V_a$ and the queue length $Q_a$ of type $a$ customers is stochastically larger than the virtual waiting time $V_b$ and queue length $Q_b$ of type $b$ customers for any given strategy profile $(F_a,F_b)$:
\begin{equation*}
V_b(t)<_{\rm st} V_a(t),\ Q_b(t)\leq_{\rm st} Q_a(t),\ \forall t\in[0,T]\ ,
\end{equation*}
which further implies that
\begin{equation*}
\E V_b(t)\leq \E V_a(t),\ \E Q_b(t)\leq \E Q_a(t),\ \forall t\in[0,T]\ .
\end{equation*}
\end{lemma}
\begin{proof}
The virtual waiting time, or workload, for type $i\in\mathcal{C}$ customers can be constructed as follows: the input process of work is a nonhomogeneous Poisson process defined by the arrival strategies, each job has a size of $X_i$, and work is continuously processed at a rate of one per unit of time. The following arguments are for a sample path constructed of the identical arrival process and coupled sequences of job sizes $\{X_{i,k}\}_{k=1}^\infty$ such that $X_{a,k}\geq X_{b,k}$ for all $k\geq 1$. Specifically, let $U_k\sim$Unif$[0,1]$ and denote
\[
X_{ki}=-\frac{1}{\mu_i}\log U_k\ ,
\]
then clearly $X_{a,k}=\frac{\mu_b}{\mu_a}X_{b,k}>X_{b,k}$ for all $k\geq 1$. For any strategy profile $(F_a,F_b)$, the virtual waiting time for a type $i\in\mathcal{C}$ arriving at $t\in\mathcal{T}$ is given by
\[
V_i(t):=A_i(t)- t+L_i(t)\ ,
\]
where $A_i(t)$ is a nonhomogeneous compound Poisson process with cumulative rate $\lambda_a F_a(t)+\lambda_b F_b(t)$ and jump sizes $X_i$, and
\[
L_i(t):=\left(-\inf_{s\in[0-,t]}\{A_i(s)- s\}-V_0\right)^{+}\ .
\]
Observe that $L_i(t)$ is a non-decreasing process that increases only when $V_i(t)=0$ and that $A_a(0)>0$ if and only if $A_b(0)>0$. Suppose that $A_a(0),A_b(0)>0$ and denote $\tau_i=\inf\{t\geq 0:\ V_i(t)=0\}$, then as $A_a(t)\geq A_b(t)$ we have that
\[
V_a(t)=A_a(t)- t \geq A_b(t)- t= V_b(t),\ t\in[0,\tau_b]\ .
\]
Furthermore, as $X_{a,k}>X_{b,k}$ for all jobs, the number of departures from queue $a$ is at most equal to the number of departures from queue $b$ during $[0,\tau_b]$ because both servers are working continuously. The common arrival time of jobs to both systems further implies that
\[
Q_a(t) \geq Q_b(t),\ t\in[0,\tau_b]\ .
\]
The process $L_b(t)$ increases after $\tau_b$ as long as $V_b(t)=0$, but clearly $0=V_b(t)\leq V_a(t)$ and $0=Q_b(t)\leq Q_a(t)$ during any such period. If a jump $A_b(t)-A_b(t-)$ occurs at $t<\tau_a$ then $V_a(t)$ has a bigger jump than $V_b(t)$ both and the dominance remains until the next time $s$ such that $V_b(s)=0$. This continues until $\tau_a$ for which
\[
V_a(\tau_a)=V_b(\tau_a)=Q_a(\tau_a)=Q_b(\tau_a)=0\ ,
\]
and both $V_i$ stay at zero until the next arrival time and the same process starts again with a new jump for both processes. The same argument is valid when $A_a(0)=A_b(0)=0$ and the process starts at the time of the first arrival. We conclude that for every sample path $V_a(t)\geq V_b(t)$ for all $t\in\mathcal{T}$ and $V_a(s)> V_b(s)$ for all $\mathcal{S}\subset\mathcal{T}$ such that $\mathcal{S}$ is a non-empty union of positive intervals, i.e., the set $\mathcal{S}$ has non-zero measure. Similarly, the same holds for $Q_a(t)$ and $Q_b(t)$. 
\end{proof}

\section{Best response algorithms}\label{sec:appn_B}

This appendix details Algorithm~1 and a necessary subroutine that we denote Algorithm~1-1. The purpose of the algorithm is to find the symmetric within type best response arrival distribution $\vc{p}_{i}$, that satisfies the equilibrium condition \eqref{eqn:NE_Arrival_Distributions}, when the customers of the other type arrive according to $\vc{p}_{-i}$. The outline is as follows: (1) finding the first time slot $\theta\in\mathcal{T}$ such that $p_{i,\theta}>0$, (2) a bisection search for the value of $p_{i,\theta}$ which uniquely defines all probabilities $p_{i,u}$ for $u>\theta$, (3) stopping the procedure when $\sum_{u=\theta}^Tp_{i,u}=1$. Algorithm~1 is designed for step (1) and checking the equilibrium condition of (3), and Algorithm~1-1 performs the bisection search of part (2).

Note that the implementation of the algorithm may be refined to make it more efficient, but the purpose of the description here is to provide a concise description. For example, in line 4 of Algorithm~1-1, we may choose $p_{i,\theta}^{(R)}$ in a slightly better way, e.g., $p_{i,\theta}^{(R)} = \min\left\{1, {\rm Equation} \ \eqref{eq:NE_p+} \right\}$, but we've chosen $p_{i,\theta}^{(R)} := 1$ for the simplicity of the presentation.

\begin{algorithm}[h]
\renewcommand{\thealgorithm}{}
\caption{\textbf{1}: Symmetric best response of type $i$ to $\vc{p}_{-i}$.}
\label{alg:best_response}
\textbf{Input}: $\vc{p}_{-i}$, $\vc{\lambda} := (\lambda_{a}, \lambda_{b})$, $\vc{x}_{i}$, $\epsilon>0$ \\
\textbf{Output}: $\vc{p}_{i}$
\begin{algorithmic}
\State init $\theta := 0$
\While{$\theta \le T$}
	\State init $\vc{p}_{i} := (0,0,\ldots,0)$
	\State compute $(w_{i,t} ; 0 \le t \le \theta)$ by \eqref{eqn:w_{i,t}}
	\State set $\ol{w}_{i} := w_{i,\theta}$, $w_{\rm min} := \min( w_{i,t} ; 0 \le t \le \theta-1 )$, where $\min (\phi) := \infty$ 
	\If{$\ol{w}_{i} < w_{\rm min}$}
		\State compute $(p_{i,t} ; \theta \le t \le T)$ by {\bf Alg.1-1}($\vc{p}_{i}, \vc{p}_{-i}, \vc{\lambda}, \vc{x}_{i}, \theta, \epsilon$)
	\Else
		\State $\theta := \theta + 1$
	\EndIf
\EndWhile
\State \textbf{return} {$\vc{p}_{i}$}
\end{algorithmic}
\end{algorithm}

\begin{algorithm}[H]
\renewcommand{\thealgorithm}{}
\caption{\textbf{1-1}: Subroutine in Algorithm 1 (Bisection search).}
\label{alg:subroutine}
\textbf{Input}: $\vc{p}_{i}$, $\vc{p}_{-i}$, $\vc{\lambda} := (\lambda_{a}, \lambda_{b})$, $\vc{x}_{i}$, $\theta$, $\epsilon > 0$\\
\textbf{Output}: $\vc{p}_{i}$, $\theta$
\begin{algorithmic}
\State init $\vc{p}_{i}^{({\rm L})} := (0,0,...,0)$, $\vc{p}_{i}^{({\rm M})} := (0,0,...,0)$, $\vc{p}_{i}^{({\rm R})} := (0,0,...,0)$
\State set $p_{i,\theta}^{({\rm R})} := 1$, $p_{i,\theta}^{({\rm M})} := (p_{i,\theta}^{(L)} + p_{i,\theta}^{({\rm R})})/2$
\While{true}
	\For{$k = {\rm L}, {\rm M}, {\rm R}$}
		\State compute $w_{i,\theta}$ by \eqref{eqn:w_{i,t}} with $\vc{p}_{i}:= \vc{p}_{i}^{(k)}$
		\State init $\ol{w}_{i}^{(k)} := w_{i,\theta}$ 
		\State compute $(p_{i,t}^{(k)} ; \theta + 1 \le t \le T)$ by \eqref{eqn:NE_Arrival_Distributions} with $\ol{w}_{i} := \ol{w}_{i}^{(k)}$, $\vc{p}_{i}:= \vc{p}_{i}^{(k)}$
	\EndFor
	\If{$1 -\epsilon < || \vc{p}_{i}^{(\rm M)} || < 1 + \epsilon$}
		\State set $\vc{p}_{i} := \vc{p}_{i}^{(\rm M)}$, $\theta := \infty$
		\State \textbf{break}
	\ElsIf{$|| \vc{p}_{i}^{(\rm L)} || > 1$}
		\State set $\theta := \theta + 1$
		\State \textbf{break}
	\ElsIf{$|| \vc{p}_{i}^{(\rm M)} || < 1$}
		\State set $p_{i, \theta}^{(\rm L)} := p_{i, \theta}^{(\rm M)}$, $p_{i, \theta}^{(\rm M)} := (p_{i, \theta}^{(\rm L)} + p_{i, \theta}^{(\rm R)})/2$
	\ElsIf{$|| \vc{p}_{i}^{(\rm M)} || > 1$}
		\State set $p_{i, \theta}^{(\rm R)} := p_{i, \theta}^{(\rm M)}$, $p_{i, \theta}^{(\rm M)} := (p_{i, \theta}^{(\rm L)} + p_{i}^{(\rm R)})/2$
	\EndIf
\EndWhile
\State \textbf{return} $\vc{p}_{i}$, $\theta$
\end{algorithmic}
\end{algorithm}


\end{appendices}


\end{document}